\documentclass[leqno]{article}
\usepackage{amsthm}
\author{Gerard Brunick\footnote{
    University of Texas at Austin,
    Department of Mathematics,
    1 University Station C1200,
    Austin, TX, 78712.
    E-mail: gbrunick@math.utexas.edu}
}
\def \keywords{\noindent \textbf{Keywords: }}
\def \subclass{\noindent
  \textbf{Mathematics Subject Classification (2010): }}

\usepackage{amsmath, amssymb, mathrsfs, mathtools}
\usepackage{enumitem, epsfig, suffix}

\numberwithin{equation}{section}
\newtheorem{theorem}{Theorem}[section]
\newtheorem{lemma}[theorem]{Lemma}
\newtheorem{corollary}[theorem]{Corollary}
\theoremstyle{definition}
\newtheorem{definition}[theorem]{Definition}
\theoremstyle{remark}
\newtheorem{remark}[theorem]{Remark}

\def \dx#1{\textstyle\frac{\partial}{\partial #1}}
\WithSuffix \def \dx [#1]#2 {\frac{\partial #1}{\partial #2}}
\def \dfx #1#2{\frac{\partial #1}{\partial #2}}

\let \minn \wedge
\let \maxx \vee
\def \cross {\mathord{\times}}
\def \Smu{S^{d_0}_\mu}

\def \atilde {\tilde a}
\def \btilde {\tilde b}

\def \dxx #1#2{\textstyle\frac{\partial^2}{\partial #1 \partial #2}}
\def \dfxx #1#2#3{\frac{\partial^2 #1}{\partial #2 \partial #3}}

\def \dil {\delta}
\def \dilbar {\bar \delta}
\def \dist{\operatorname{dist}}

\def \Ahat {\widehat A}
\def \Atilde {\widetilde A}
\def \ahat {\widehat a}
\def \atilde {\widetilde a}
\def \btilde {\widetilde b}
\def \bhat {\widehat b}
\def \Bhat {\widehat B}
\def \chat {\widehat c}
\def \ctilde {\widetilde c}
\def \Chat {\widehat C}
\def \dbar {{\bar d}}
\def \fhat {\widehat f}
\def \hhat {\widehat h}

\def \Lhat {\widehat L}

\def \Phat{\widehat {\bb P}}

\def \xbar {{\bar x}}
\def \ybar {{\bar y}}
\def \zbar {{\bar z}}

\def \muhat {\widehat \mu}

\def \sigmahat {\widehat \sigma}
\def \sigmatilde {\widetilde \sigma}

\def \bb {\mathbb}
\def \cond {\mid}
\def \scr {\mathscr}

\def \diff {d}
\def \half {{\tfrac{1}{2}}}

\def \defeq {=}
\def \cross {\mathord \times}

\DeclareMathAlphabet{\mathbbnumber}{U}{bbold}{m}{n}
\def \ind #1{\mathbbnumber{1}_{#1}}
\def \indd #1{\mathbbnumber{1}_{\{#1\}}}
\def \inv {^{-1}}
\def \tr{\operatorname{tr}}
\def \op {\circ}

\begin{document}

\title{Uniqueness in Law for a Class of Degenerate
       Diffusions with Continuous Covariance}

\maketitle

\begin{abstract}
  We study the martingale problem associated with the operator
  \begin{equation*}
    L u = \partial_s u(s,x)
        + \frac{1}{2} \sum_{i,j=1}^{d_0} a^{ij} \partial_{ij} u(s,x)
        + \sum_{i,j=1}^d B^{ij} x^j \partial_i u(s,x),
  \end{equation*}
  where $d_0 \leq d$.  We show that the martingale problem is
  well-posed when the function $a$ is continuous and strictly
  positive-definite on $\bb R^{d_0}$ and the matrix $B$ takes a
  particular lower-diagonal, block form.  We then localize this result
  to show that the martingale problem remains well-posed when $B$ is
  replaced by a sufficiently smooth vector field whose Jacobian matrix
  satisfies a nondegeneracy condition.
  \medskip

  \\
  \subclass{60H10, 35K65}
\end{abstract}

\section{Introduction} \label{sec:1}

In this paper we consider stochastic differential equations (SDEs) of
the form
\begin{equation} \label{eq:1.1}
  \begin{split}
    \diff X_t &= b(t, X_t) \, \diff t + \sigma(t, X_t) \, \diff W_t,
  \end{split}
\end{equation}
where the process $X$ takes values in $\bb R^d$ and $W$ is a Brownian
motion of dimension $d_0 \leq d$.  We provide conditions which are
sufficient to ensure that the solution to this SDE is unique in law
when the covariance function $a = \sigma \sigma^T$ is degenerate and
continuous.

When $d_0 = d$, the drift is bounded, and the covariance function is
bounded and uniformly positive definite on compact sets, then a number
of sufficient conditions are available which ensure uniqueness in law
for the SDE~\eqref{eq:1.1}.  Stroock and Varadhan
\cite{stroock1979mdp} have shown that weak uniqueness holds when the
covariance is continuous in the spacial variables, and this continuity
is uniform on compact time sets.  More recently, Bramanti and
Cerutti~\cite{bramanti1993scpp} have provided an estimate which
implies that the solution is unique when the covariance function is
VMO continuous in space and time, and Krylov~\cite{krylov2007peev} has
relaxed this condition to VMO-continuity in the spacial variables
only.

If we retain the assumption that the covariance is uniformly positive
definite on compact sets and we further assume that the covariance is
a function of the spacial variables only, then more results are
available.  Krylov~\cite{krylov1969dpri} has shown that uniqueness
holds for all measurable covariance functions when $d \leq 2$.  Bass
and Pardoux~\cite{bass1987udpc} show that uniqueness holds when $\bb
R^d$ can divided into a finite number of polyhedrons such that the
covariance function is constant on each polyhedron.  Cerutti et
al.~\cite{cerutti1991usdd} show that uniqueness holds when the
covariance is continuous outside of a countable set that has a single
cluster point.  Gao~\cite{gao1993mpdo} shows that uniqueness holds
when the covariance function is continuous on the sets $\{x \in \bb
R^d : x^1 > 0\}$ and $\{x \in \bb R^d : x^1 \leq 0\}$.
Safonov~\cite{safonov1994wuse} shows that weak uniqueness holds when
the set of discontinuities of the covariance function has zero
$\alpha$-Hausdorff measure for sufficiently small $\alpha$.
Krylov~\cite{krylov2004wusd} gives a number of results which may be
combined to produce weak uniqueness results for a variety of settings.
Finally, Nadirashvili \cite{nadirashvili1997nmpd} provides a
counterexample which shows that uniqueness may not hold if the
covariance function is only assumed to be measurable and $d \geq 3$.

In the case of multidimensional diffusions with degenerate covariance,
fewer results are available.  It is a classical result that pathwise
uniqueness holds if the coefficients $\sigma$ and $b$ are Lipschitz
continuous.  Figalli \cite{figalli2008eums} has shown that uniqueness
holds for the associated Stochastic Lagrangian Flow when the
covariance is a bounded, deterministic function of time and the drift
is a BV vector field whose divergence is controlled.  Le~Bris and
Lions~\cite{le2008eusf} provide estimates for the forward equation
associated with the SDE~\eqref{eq:1.1}, and they sketch how these
results may be used to produce weak uniqueness results for SDEs whose
coefficients possess sufficient Sobolev regularity.

In contrast to the results just mentioned, we consider a setting where
the null space of the covariance may be nontrivial everywhere and the
covariance is only assumed to be a continuous function time and space.
We are able to obtain weak uniqueness results in this setting by
imposing conditions on the drift which ensure that the process is, in
some sense, locally hypoelliptic.

We will delay the precise statement of our results to
Section~\ref{sec:5}, and instead give two examples that illustrate
the kinds of SDEs that can be handled.  To present the first example,
suppose that $d = n d_0$ for some $n \geq 2$ and let $X_t = (X^1_t,
\dots, X^d_t)$.  We then define the SDE:
\begin{equation} \label{eq:1.2}
  \left\{
  \begin{alignedat}{2}
    X^i_t &= \bhat^i(t, X_t) \, \diff t
             + \sum_{j=1}^{d_0} \sigmahat^{ij}(t,X_t) \, \diff W^j_t,
    &\qquad &1 \leq i \leq d_0,
    \\
    X^i_t &= X^{i-d_0}_t \, \diff t, &\qquad &d_0 < i \leq d,
  \end{alignedat}\right.
\end{equation}
where $W$ is a $d_0$-dimensional Brownian motion and the
functions $\bhat$ and $\sigmahat$ may depend upon all of the
components of the process $X$.  Notice that if we rewrite the
equation~\eqref{eq:1.2} in the form~\eqref{eq:1.1}, then $\sigma
\sigma^T$ is of rank $d_0 < d$ everywhere.  Theorem~\ref{thm:5.10}
asserts that existence and uniqueness in law hold for the
SDE~\eqref{eq:1.2} when $\bhat$ and $\sigmahat$ satisfy a linear
growth condition and $\sigmahat \sigmahat^T$ is continuous and
strictly positive definite on $\bb R^{d_0}$.

We can also handle a situation where the drift of the finite variation
components of $X$ is given by a sufficiently smooth function that
satisfies a local nondegeneracy condition.  To state this example,
fix $d_0 \geq d/2$, and write $x \in \bb R^d$ and $X$ in the form $x =
(x', x'')$ and $X = (X', X'')$, where the first coordinate denotes the
first $d_0$ components, and the second component denotes the remaining
$d - d_0$ components.  Now consider the following SDE written in
vector form:
\begin{equation} \label{eq:1.3}
  \left\{
  \begin{aligned}
    X'_t &= b'(t, X_t) \, \diff t + \sigmatilde(t,X_t) \, \diff W_t,
    \\
    X''_t &= b''(t,X_t) \, \diff t,
  \end{aligned}\right.
\end{equation}
where $W$ is a $d_0$-dimensional Brownian motion, $b'$ takes values in
$\bb R^{d_0}$, $\sigmatilde$ takes values in the space of $d_0 \cross
d_0$-matrices, and $b''$ takes values in $\bb R^{d-d_0}$.  We now
assume that all of the coefficients satisfy a linear growth condition,
$\sigmatilde$ is continuous, and $b'' \in C^2$.  We also need to
impose nondegeneracy conditions on both $\sigmatilde$ and $b''$. We
assume that $\sigmatilde \sigmatilde^T$ is strictly positive definite,
and we assume that the Jacobian matrix of $b''$ with respect to the
variables $x'$ is of rank $d - d_0$ at each point.  Under these
conditions, it follows from Theorem~\ref{thm:5.14} that existence and
uniqueness in law hold for the SDE~\eqref{eq:1.3}.

To obtain these results, we follow the approach developed by Stroock
and Varadhan \cite{stroock1969dpcc1, stroock1969dpcc2,
  stroock1979mdp}.  We first produce a Calder{\'o}n-Zygmund-type
estimate for the solutions of Kolmogorov's backward equation
\begin{equation} \label{eq:1.4}
  \partial_s u + L^{a,B} u = f,
\end{equation}
where
\begin{equation} \label{eq:1.5}
  L^{a,B} u(s,x) =
  \frac{1}{2} \sum_{i,j=1}^{d_0} a^{ij}(s,x) \partial_{ij} u(s,x)
  + \sum_{i,j=1}^d B^{ij} x^j \partial_j u(s,x),
\end{equation}
with $d_0 \leq d$ and $B$ is a fixed matrix that satisfies a
structural condition given in Section~\ref{sec:2}.  We then make a
perturbation argument to produce a local uniqueness result, followed
by a localization argument to produce a global result.

Before we close the introduction, we should mention that
equation~\eqref{eq:1.4} has been studied rather extensively, and we
will not attempt to give a comprehensive account of the literature.
Instead, we refer the reader to the survey
article~\cite{lanconelli2002lnue} and we mention only two references.
Lanconelli and Polidoro~\cite{lanconelli1994cheo} identify a
homogeneous group with respect to which the operator $\partial_s +
L^{a,B}$ is left-translation invariant when $a$ is constant.  We make
extensive use of this group structure in everything that follows.
Bramanti, Cerutti, and Manfredini~\cite{bramanti1996les} give
estimates for solutions of \eqref{eq:1.4} when the $a$ is a
VMO-continuous function of space and time.  These results are obtained
by combining estimates from the constant coefficient case with deep
results about the commutators of singular integrals on homogeneous
spaces from~\cite{bramanti1996csih}.  While these result are in many
ways more sophisticated than the approach that we take in
Section~\ref{sec:4}, they do not imply the estimates that we obtain.
In particular, we study the case where $a$ is a measurable function of
time only, and coefficients in this class need not be
VMO-continuous. Moreover, there are some technical challenges which
must be overcome before the estimates obtained in
\cite{bramanti1996les} may be used to obtain weak uniqueness results
for SDEs with discontinuous and degenerate covariance.  We refer the
reader to Remark~\ref{thm:5.5} for a more detailed discussion of the
issue that arises.

The outline of the paper is as follows.  In Section~\ref{sec:2} we
introduce notation.  In Section~\ref{sec:3} we study the transition
function which will play the role of a fundamental solution for the
equation~\eqref{eq:1.4}.  In Section~\ref{sec:4}, we derive the
$L^p$-estimate upon which our local uniqueness result depends, and in
Section~\ref{sec:5} we provide the announced uniqueness results.

\section{Notation and Geometric Structure} \label{sec:2}

We let $|\cdot|$ denote the Euclidean norm with associated inner
product $\langle \cdot, \cdot \rangle$.  We use superscripts to access
the components of a vector and we start numbering our components at
zero when the first coordinate corresponds to time.  We let $\bb
M^{d_0\cross d_1}$ denote the set of $d_0\cross d_1$ matrices and
$\|\cdot\|$ denotes the operator norm on matrices which is compatible
with the Euclidean norm.  We abbreviate $\bb M^{d \cross d}$ to $\bb
M^{d}$ and let $I^d \subset \bb M^{d}$ denote the identity matrix.  We
let $B^d_r(x) \subset \bb R^d$ denote the open ball of radius $r$
centered at $x$ and $\overline B^d_r(x)$ denotes the closed ball.  We
let $S^d_+ \subset \bb M^d$ denote the symmetric, nonnegative-definite
matrices, we write $A \geq B$ if $A - B \in S^d_+$, and we let
$S^d_\mu \subset S^d_+$ denote the matrices whose eigenvalues are
contained in the interval $[1/\mu, \mu]$ when $\mu \geq 1$.

We let $C_K(\bb R^d)$ denote the continuous functions with compact
support, we set $\bb R_+ = [0,\infty)$, and we let $C^n(\bb R_+\cross
\bb R^d)$ denotes the class of functions that possess $n$ continuous
derivatives in $(0,\infty)\cross \bb R^d$, each of which admits a
continuous extension to $\bb R_+\cross \bb R^d$.  If $\alpha$ is a
multiindex, then $D^\alpha f$ denote the partial derivative
corresponding to $\alpha$.  If no multiindex is given, then $Df$
denote the gradient of a scalar function and the Jacobian matrix of a
vector-valued function, and $D^2f$ denotes the Hessian of a scalar
function.  If the components of $\bb R^d$ have been partitioned as $x
= (y,z)$, then $D_y$ denotes the gradient or Jacobian matrix
restricted to the components in $y$.  We will also use
$\dxx{x^i}{x^j}$, $\partial_s$, and $\partial_{ij}$ to denote partial
derivatives, but we will never use subscripts.

We let $X$ denote the canonical process on $C(\bb R_+; \bb R^{d})$,
and we equip this space with the locally uniform topology.  We let
$\scr C^d$ denote the Borel $\sigma$-field on $C(\bb R_+; \bb R^{d})$,
we set $\scr C_t^d = \sigma(X_s : s \leq t)$, and we set $\bb C^d =
(\scr C_t^d)_{t\geq 0}$.  The filtration $\bb C^d$ does not satisfy
the usual conditions of right-continuity and completeness, but this
will not cause any problems in what follows.  If $Y$ is an $\bb
R^d$-valued process, $f : \bb R_+ \cross \bb R^d \rightarrow \bb R^r$
is measurable, and $B \in \bb M^{r\cross d}$, then $f(Y)$ denotes the
process $t \mapsto f(t,Y_t)$ and $BY$ denotes the process $t \mapsto B
Y_t$.  The following definition is convenient when dealing with
processes whose covariance is degenerate.

\def \tsum{{\textstyle\sum}}
\begin{definition}
  Let $d_0 \leq d$, let $(\Omega, \scr F, \bb F = (\scr F_t)_{t \geq
    0}, \bb P)$ be a filtered probability space, let $a : \bb R_+
  \cross \Omega \rightarrow S^{d_0}_+$ and $b : \bb R_+ \cross \Omega
  \rightarrow \bb R^d$ be $\bb F$-progressive processes, and let
  $\Lhat^{a,b}$ denote the stochastic operator
  \begin{equation*}
    \Lhat^{a,b} f(s,x)
    \defeq \partial_s f(s,x)
	   + \frac{1}{2} \sum_{i,j=1}^{d_0} a^{ij}_s \, \partial_{ij} f(s,x)
	   + \sum_{i=1}^d b^i_s \, \partial_i f(s,x).
  \end{equation*}
  We say that a continuous, $\bb F$-adapted, $\bb R^d$-valued process
  $Y$ is a solution to the $(a,b)$-martingale problem starting at
  $(s,x) \in \bb R_+ \cross \bb R^d$ if $\bb P(Y_t = x, \; \forall t
  \leq s) = 1$, and the process $M_t = f(t, Y_t) - \int_s^t
  \Lhat^{a,b}f(u, Y_u) \, \diff u$ is a martingale on $[s,\infty)$ for
  each $f \in C^\infty_K(\bb R_+\cross \bb R^d)$.

  If the processes $a$ and $b$ are defined on $(C(\bb R_+; \bb R^d),
  \scr C^d, \bb C^d)$, then we say that a probability measure $\bb P$
  on $C(\bb R_+; \bb R^d)$ is a solution to the martingale problem if
  the canonical process is a solution to the martingale problem under
  the measure $\bb P$, and we say that the $(a,b)$-martingale problem
  is well-posed if there exists a unique measure which solves the
  $(a,b)$-martingale problem for each initial condition.
\end{definition}

\def \oh#1#2{O} 

We now give a brief description of the geometric setting in which we
shall be working.  The reader may consult \cite{lanconelli1994cheo} or
\cite{bramanti1996les} for a more thorough discussion.  Fix some $d
\geq 1$ and let $B \in \bb M^{d}$ denote a matrix which takes the
following lower-triangular, block form
\begin{equation*}
  B =
  \begin{bmatrix}
     0 & 0 & \dots	& 0 & 0 \\
     B_1     & 0 & \dots	& 0 & 0 \\
     0 & B_2     & \dots	& 0 & 0 \\
     \vdots  & \vdots  & \ddots & \vdots      & \vdots	\\
     0 & 0 & \dots	& B_n	      & 0 \\
  \end{bmatrix},
\end{equation*}
where $( d_0, \dots, d_n )$ is a nonincreasing sequence with $\sum_i
d_i = d$, $B_i\subset \bb M^{d_{i} \cross d_{i-1}}$ is a matrix of
rank $d_i$, and $0$ denotes a matrix of zeros whose dimensions may
vary with each appearance.  It follows from this block structure that
$B^i = 0$ when $i > n$.

Once a matrix $B$ has been fixed, we define the following
binary operation on $\bb R\cross \bb R^d$:
\begin{equation*}
  (s,x) \op (t,y) \defeq (s+t, e^{tB} x + y).
\end{equation*}
It is easy to check that $(\bb R^{1+d}, \,\op\,)$ is a group with
identity element $(0,0)$ whose inverse operation is given by
\begin{equation*}
  (s,x)\inv = (-s, -e^{-tB} x).
\end{equation*}
The reader can also check that the operator $\partial_s + L^{a,B}$ is
left-translation invariant with respect to this group when $L^{a,B}$
is defined as in~\eqref{eq:1.4} and $a$ is constant.

Now let $\dilbar_\lambda \in \bb M^{1+d}$ denote the diagonal matrix
\begin{equation*}
  \begin{bmatrix}
    \lambda^2 & 0 & 0 & \dots & 0 \\
    0 & \lambda I^{d_0} & 0 & \dots & 0 \\
    0 & 0 & \lambda^3 I^{d_1} & \dots & 0 \\
    \vdots & \vdots & \vdots & \ddots & \vdots \\
    0 & 0 & 0 & \dots & \lambda^{2n+1} I^{d_n} \\
  \end{bmatrix},
\end{equation*}
and let $\dil_\lambda \in \bb M^{d}$ denote the matrix obtained by
removing the first row and column from $\dilbar_\lambda$.  It follows
easily from the block structure of $B$ that $\dil_\lambda B^i
\dil_\lambda\inv = \lambda^{2i} B^i$.  As $B$ is nilpotent, we have
\begin{equation*} 
  \dil_\lambda e^{tB} \dil_\lambda\inv = e^{\lambda^2 t B}.
\end{equation*}
One can then check that each dilation $\dilbar_\lambda$ is an
automorphism of the group $(\bb R^{1+d}, \,\op\,)$, so the collection
$(\bb R^{1+d}, \,\op\,, \dilbar)$ forms a homogeneous group in the
sense of Definition XIII.5.2 in \cite{stein1993harm}.  We then set
$\dbar = 2 + d_0 + 3 d_1 + \dots + (2n+1) d_n$, so $\det \dilbar_\lambda
= \lambda^{\dbar}$ and $\dbar$ gives the homogeneous dimension of the
group $(\bb R^{1+d}, \,\op\,, \dilbar)$.

Finally, let
\begin{equation*}
  \rho(\xbar)
  = \inf\bigl\{ \lambda > 0 : |\dilbar_\lambda\inv \xbar| \leq 1 \bigr\},
  \qquad \xbar \in \bb R^{1+d},
\end{equation*}
denote the homogeneous norm associated with these dilations and
observe that $\rho(\dilbar_\lambda \xbar) = \lambda \rho(\xbar)$ for
all $\xbar \in \bb R^{1+d}$.  Moreover, $\dilbar_\lambda \leq \lambda
I^{1+d}$ when $\lambda \leq 1$ and $\lambda I^{1+d} \leq
\dilbar_\lambda$ when $\lambda \geq 1$, so it follows immediately that
$|\xbar| \leq \rho(\xbar)$ when $|\xbar| \leq 1$ and $\rho(\xbar) \leq
|\xbar|$ when $|\xbar| \geq 1$.

In the following sections, we will always assume that $B$ is a matrix
which satisfies the structural conditions given above.  We also
observe that the matrix $B$ fully determines the constants $d$,
$\dbar$, $n$, and $(d_0, \dots, d_n)$, the binary operation $\circ$,
and the matrix $\dilbar_\lambda$.

\section{Initial Estimates for a Transition Function} \label{sec:3}

We now begin to study the $(c, BX)$-martingale problem on $C(\bb
R_+; \bb R^d)$ when $c : \bb R_+ \rightarrow S^{d_0}_\mu$ for some
$\mu \geq 1$ and $B$ satisfies the structural conditions given in
Section~\ref{sec:2}.  Let $\sigma(t) = \sqrt c(t)$ denote the
symmetric, positive-definite square root of $c(t)$, and consider the
vector-valued SDE
\begin{equation} \label{eq:3.1}
  \left\{
  \begin{alignedat}{2}
    \diff Y_t(s,x) &= BY_t \, \diff t + \begin{bmatrix}
                                          \sigma(t) \\
                                          0
                                        \end{bmatrix} \diff W_t,
                   &\quad &t > s,
     \\
     Y_t(s,x) &= x, &\qquad &t \leq s,
  \end{alignedat}
  \right.
\end{equation}
where $W$ is a $d_0$-dimensional Brownian motion and $0 \in \bb
M^{(d-d_0)\cross d_0}$.  A solution to this equation is given by
\begin{equation*}
  Y_t(s,x) = e^{(t-s)^+ B} x
	     + \int_s^{s \vee t} e^{(t-u)B} \begin{bmatrix}
					      \sigma(u) \\
					      0
					    \end{bmatrix} \, dW_u.
\end{equation*}
The coefficients in equation~\eqref{eq:3.1} are Lipschitz continuous
in space, so this solution is both pathwise unique and unique in
law. In particular, if we let $\bb P_{s,x}$ denote the law of the
process $Y(s,x)$, then we see that $\bb P_{s,x}$ is the unique
solution to the $(c, BX)$-martingale problem starting at $(s,x)$, and
the collection of measures $\{\bb P_{s,x}\}_{(s,x)\in \bb R_+\cross \bb
  R^d}$ form a strong Markov family on $(C(\bb R_+; \bb R^d), \bb
C^d)$.  The transition function associated with this strong Markov
family is given by
\begin{equation} \label{eq:3.2}
 \begin{split}
  p_{c,B}(s, x; t, y)
  &= \indd{t>s} \, (2 \pi)^{-d/2} \, \bigl\{\det C_{c,B}(s,t)\bigr\}^{-1/2}
     \\ &\qquad
     \times \exp\Bigl\{ -\frac{1}{2}
			\bigl\langle C_{c,B}\inv(s,t) (y - e^{(t-s)B}x),\;
				     y - e^{(t-s)B}x \bigr\rangle
		 \Bigr\},
  \end{split}
\end{equation}
where
\begin{equation} \label{eq:3.3}
  C_{c,B}(s,t) = \int_{s}^{s \maxx t} e^{(t-u)B}
		\begin{bmatrix}
		  c(u) & 0 \\
		  0 & 0
		\end{bmatrix} e^{(t-u)B^T} \, du.
\end{equation}

We will now study this transition function using analytic tools.
This will be more pleasant if the transition function is defined on
the entire time line, so we will now assume that $c : \bb R
\rightarrow S^{d_0}_\mu$ and we will define $p_{c,B}(s,x;t,y)$ for all
$(s,x;t,y) \in \bb R^{2(1+d)}$.  In Lemma~\ref{thm:3.2}, we will
see that $C_{c,B}(s,t)$ is invertible when $t > s$, so $p_{c,B}$ is
well-defined. The fact that $C_{c,B}(s,t)$ is invertible for all $s >
t$ reflects the fact that the backward equation associated with the
process $Y(s,x)$ is hypoelliptic when $c$ is constant.  We now define
the Green's operator associated with this transition function:
\begin{equation*}
  G_{c,B} f(\bar x) = \int_{\bb R^{1+d}} p_{c,B}(\bar x; \bar y) f(\bar y)
					 \, \diff \bar y,
  \qquad \xbar \in \bb R^{1+d}, f \in C_K(\bb R^{1+d}).
\end{equation*}
We will also need the operator
\begin{align*}
  L^{c,B} u(s,x)
  &= \frac{1}{2} \sum_{i,j=1}^{d_0} c^{ij}(s)
			     \partial_{ij}u(s,x)
     + \bigl\langle Bx, D_x u(s,x) \bigr\rangle.
\end{align*}
If we apply the operator $L^{c,B}$ to the transition function
$p_{c,B}(\xbar; \ybar)$, then we will add a subscript to indicate
which set of coordinates the operator should act upon.

We first give a number of scaling properties possessed by the
transition function.  These properties follow easily from the fact
that $e^{\lambda t B} = \dil_\lambda^{1/2} e^{t B}
\dil_\lambda^{-1/2}$ and $\det \dil_\lambda = \lambda^{\dbar-2}$, so
we leave their verification to the reader.
\begin{lemma} \label{thm:3.1} Let $c : \bb R \rightarrow \Smu$ be a
  measurable function, let $\xbar = (s,x) \in \bb R^{1+d}$, $\ybar =
  (t,y) \in \bb R^{1+d}$, $\zbar = (u,z) \in \bb R^{1+d}$, let
  $\lambda > 0$, and set $c_1(\tau) = c(u + \tau)$ and $c_2 (\tau) =
  c(\lambda^2 \tau)$.  Then
  \begin{align}
    \label{eq:3.4}
    p_{c,B}(\zbar\op\xbar; \; \zbar\op\ybar)
    &= p_{c_1}(\xbar; \ybar),
    \\
    \label{eq:3.5} 
    C_{c,B}(\lambda^2 s, \lambda^2 t)
      &= \dil_\lambda C_{c_2} (s, t) \dil_\lambda,
    \\
    \label{eq:3.6} 
    p_{c,B}(\dilbar_\lambda \xbar; \dilbar_\lambda \ybar)
      &= \lambda^{2-\dbar} p_{c_2}(\xbar; \ybar).
  \end{align}
\end{lemma}

We will soon need bounds on the transition function which only depend
upon $c$ through $\mu$, so we will now develop some lemmas for uniformly
dominating the transition function.

\begin{lemma} \label{thm:3.2} Let $c : \bb R \rightarrow \Smu$ be
  measurable.  Then $C_{c,B}(s,t)$ is invertible when $t > s$ and there
  exist polynomials $P$ and $Q$ with positive coefficients such that
  $\|C_{c,B}(s,t)\| \leq P(t-s)$ and $\|C\inv_{c,B}(s,t)\| \leq Q(1/\{t-s\})$
  when $t > s$.	 Moreover, the coefficients of the polynomials $P$ and
  $Q$ only depend upon $B$ and $\mu$.
\end{lemma}
\begin{proof}
  Set $A = \begin{bmatrix} I^{d_0} & 0 \\ 0 & 0 \end{bmatrix} \in \bb
  M^d$ and define
  \begin{equation*}
    \Chat(\tau) \defeq \int_0^{0 \maxx \tau}
      e^{(\tau-u)B} A e^{(\tau-u)B^T} \, du.
  \end{equation*}
  Then $\Chat(\tau)$ is
  strictly positive definite when $\tau > 0$.  This follows from the
  fact that $A (B^T)^i x = 0$ for all $i \geq 0$ if and only if $x =
  0$.  The reader may consult Proposition~A.1 of
  \cite{lanconelli1994cheo} for the details of this argument.  It
  follows from the definition of $C_{c,B}(s,t)$ that $\Chat(t-s) / \mu
  \leq C_{c,B}(s,t) \leq \mu \Chat(t-s)$ with respect to the natural
  partial ordering on symmetric matrices.  As a result, we see that
  $C_{c,B}(s,t)$ is invertible and $C_{c,B}\inv(s,t) \leq \mu
  \Chat\inv(t-s)$ when $t > s$.  Using the fact that $\dil_\lambda
  e^{tB} \dil_\lambda\inv = e^{\lambda^2 t B}$, one can easily check
  that $\Chat(t) = \dil_{t}^{1/2} \Chat(1) \dil_{t}^{1/2}$.  Finally,
  we recall that the Euclidean operator norm of a symmetric,
  positive-definite matrix is equal to the largest eigenvalue of the
  matrix.  Combining these observations, we see that
  \begin{align*}
    \|C_{c,B}(s,t)\| \leq \mu \|\Chat(t-s)\|
    &\leq \mu \|\Chat(1)\| \| \dil_{(t-s)}^{1/2} \|^2
    \\
    &\leq \mu \|\Chat(1)\| \bigl\{(t-s) + (t-s)^{2n+1}\bigr\},
  \end{align*}
  where $n$ denote the constant that appears in Section~\ref{sec:2}.
  As $\|\Chat(1)\|$ and $n$ are fully determined by $B$, we have
  produced the polynomial $P$.	Arguing in the same way, we see that
  \begin{align*}
    \|C\inv_{c,B}(s,t)\|
    &\leq \mu \|\Chat\inv(1)\| \| \dil_{(t-s)}^{-1/2} \|^2
    \\
    &\leq \mu \|\Chat\inv(1)\| \bigl\{(t-s)\inv + (t-s)^{-(2n+1)}\bigr\}.
  \end{align*}
  This gives the polynomial $Q$ and completes the proof.
\end{proof}

\begin{lemma} \label{thm:3.3}
  Let $0 < a \leq b$ and let $c : \bb R \rightarrow \Smu$ be measurable.
  Then there exist constants $N = N(B,\mu)$ and $\varepsilon =
  \varepsilon(B,\mu) > 0$ such that $ p_{c,B}(s,x; t,y) \leq N \exp
  \bigl\{ N |x| |y| - \varepsilon (|x|^2 + |y|^2) \bigr\} $ when $t-s
  \in [a,b]$.
\end{lemma}
\begin{proof}
  Let $\Chat$ be defined as in Lemma~\ref{thm:3.2}, let $\lambda_1 >
  0$ denote the smallest eigenvalue of $\Chat\inv(b)$, and set $\delta
  = \inf \{ |e^{s B} x| : s \in [a,b], \, x \in \bb R^d, \, |x| = 1
  \}$.	As this infimum is achieved and $e^{sB}$ is invertible,
  we have $\delta > 0$.	 When $t-s \in [a,b]$, we have
  $\Chat\inv(b) / \mu \leq
  C_{c,B}\inv(s,t) \leq \mu \Chat\inv(a)$ and
  \begin{align*}
    \lefteqn{\langle C_{c,B}\inv(s,t) (y - e^{(t-s)B}x), \,
				      y - e^{(t-s)B}x \rangle}
    \qquad \qquad \\
    &\geq \lambda_1 |y|^2 / \mu + \lambda_1 \delta^2 |x|^2 / \mu
	  - 2 \mu \|\Chat\inv(a)\| |y| e^{(b-a)\|B\|} |x|.
  \end{align*}
  In particular, if we set $N_1 = \mu \|\Chat\inv(a)\| e^{(b-a)\|B\|}$
  and $\varepsilon = \lambda_1(1 \minn \delta^2) / (2\mu)$,
  then we have
  \begin{align*}
    p_{c,B}(s,x;t,y)
    &\leq (2\mu\pi)^{d/2} \{ \det \Chat(a)\}^{-1/2}
	  \exp\{ N_1 |x| |y| - \varepsilon (|x|^2 + |y|^2)\},
  \end{align*}
  when $t-s \in [a,b]$.
\end{proof}

\begin{lemma} \label{thm:3.4} %
  Let $c : \bb R \rightarrow \Smu$ be measurable and let $\alpha,
  \beta \in (\{0\}\cup \bb N)^d$ be multiindices.  Then there exists a
  polynomial $P$ in four variables with positive coefficients such
  that
  \begin{equation*}
    |D^\alpha_x D^\beta_y p_{c,B}(s,x;t,y)|
      \leq P(t-s, 1/(t-s), |x|, |y| \bigr)
	   \; p_{c,B}(s,x;t,y),
  \end{equation*}
  when $t > s$.  If we further assume that $c \in C^\infty(\bb R;
  \Smu)$, then $\partial_s p_{c,B}(s,x;t,y) + L^{c,B}_{\xbar}
  p_{c,B}(s,x;t,y) = 0$ when $s < t$, and we may find polynomials $Q$
  and $R$ of the same form as $P$ such that
  \begin{align*}
    |\partial_s D^\alpha_x D^\beta_y \, p_{c,B}(s,x;t,y)|
      &\leq  Q(t-s, 1/(t-s), |x|, |y| \bigr)
	   \; p_{c,B}(s,x;t,y),
    \\
    |\partial_t D^\alpha_x D^\beta_y \, p_{c,B}(s,x;t,y)|
      &\leq  R(t-s, 1/(t-s), |x|, |y| \bigr)
	   \; p_{c,B}(s,x;t,y),
  \end{align*}
  when $t > s$.
  Moreover, the coefficients of $P$, $Q$, and $R$ may be chosen so
  that they only depend upon $B$, $\mu$, $\alpha$ and $\beta$.

\end{lemma}
\begin{proof}
  Define the vector-valued functions
  \begin{equation*}
      f_m(s,x;t,y) \defeq e^{m(t-s)B^T} C\inv_c(s,t) (y - e^{(t-s)B} x),
      \qquad m \in \{0,1\},
  \end{equation*}
  the matrix-valued functions
  \begin{equation*}
    g_{mn}(s,t) \defeq e^{m(t-s)B^T} C\inv_c(s,t) e^{n(t-s)B},
    \qquad m,n \in \{0,1\},
  \end{equation*}
  and the scalar functions
  \begin{align*}
    h_0(s,x;t,y) &\defeq
      \frac{1}{2} \sum_{i,j=1}^{d_0} c^{ij}(t) (f^i_0 f^j_0-g^{ij}_{00})
      + \langle By, f_0 \rangle,
    \\
    h_{1}(s,x;t,y) &\defeq
      \frac{1}{2} \sum_{i,j=1}^{d_0} c^{ij}(s) (g^{ij}_{11} - f^i_1 f^j_1)
      - \langle Bx, f_1\rangle,
  \end{align*}
  where the arguments of $f_m$ and $g_{mn}$ have been suppressed.  We
  then see that $D_y p_{c,B} = - f_0 \, p_{c,B}$, $D_x p_{c,B} = f_1
  \, p_{c,B}$, $D_y f_m = g_{m0}$, and $D_x f_m = - g_{m1}$ when $t >
  s$.  If $c$ is smooth, then we also have $\partial_t p_{c,B} = h_0
  \, p_{c,B}$ and $\partial_s p_{c,B} = h_1 \, p_{c,B}$.  Using these
  expressions, one can check directly that $\partial_s
  p_{c,B}(s,x;t,y) + L^{c,B}_{\xbar} p_{c,B}(s,x;t,y) = 0$ when $s <
  t$.

  An inductive argument using the product rule for differentiation
  shows that is enough to check that $f_n$, $g_{nm}$, and $h_n$ are
  dominated by polynomials of the required form for $n,m \in \{0,1\}$.
  Examining the expressions above, we then see that it is actually
  enough to bound $\|c(\tau)\|$, $\|e^{(t-s)B}\|$ and
  $\|C\inv_c(s,t)\|$ by polynomials of the desired form.  We have
  $\|c(\tau)\| \leq \mu$ for all $\tau$.  As the matrix $B$ is
  nilpotent, the expression $\|e^{(t-s)B}\|$ is bounded by a
  polynomial in the single variable $(t-s)$ whose coefficients are
  determined by $B$.  Finally, Lemma~\ref{thm:3.2} asserts that
  $\|C\inv_c(s,t)\|$ is bounded by a polynomial in the single variable
  $(t-s)\inv$ whose coefficients are determined by $B$ and $\mu$, so
  the proof is complete.
\end{proof}

We will need the following cancellation properties in the next section.
\begin{lemma} \label{thm:3.5} %
  Let $c : \bb R \rightarrow \Smu$ be measurable, let $1 \leq i,j \leq
  d$, and let $s < t$.	Then
    $\int_{\bb R^d} \dxx{x^i}{x^j}{p_{c,B}}(s,x;t,y) \, dx = 0$
  and
    $\int_{\bb R^d} \dxx{x^i}{x^j}{p_{c,B}}(s,x;t,y) \, dy = 0$.
\end{lemma}
\begin{proof}
  Let $f_1$ and $g_{11}$ be defined as in the last lemma, so
  $D^2_x \, p_{c,B} = (f_1 f_1^T - g_{11}) \, p_{c,B}$.	 It follows
  from Lemma~\ref{thm:3.3} and Lemma~\ref{thm:3.4} that
  \begin{equation*}
    \lim_{x^i \rightarrow \pm \infty}
    \dfx{p_{c,B}}{x^j} (s,x^1,\dots,x^i,\dots,x^d;t,y) = 0,
  \end{equation*}
  so $\int_{\bb R^d}
  \dfxx{p_{c,B}}{x^i}{x^j} (s,x;t,y) \, dx = 0$ by the Fundamental Theorem
  of Calculus and Fubini's Theorem.  To handle the second integral, we
  let $Y$ denote a $d$-dimensional, normally distributed random
  variable on some probability space with mean zero and covariance
  $C_{c,B}(s,t)$.  With this notation, we have
  \begin{align*}
    \lefteqn{\int_{\bb R^d} D^2_x p_{c,B}(s,x;t,y) \, dy }
    \qquad \qquad \\
      &= \bb E\Bigl[e^{(t-s)B^T}C\inv_c(s,t) Y Y^T C\inv_c(s,t)e^{(t-s)B}\Bigr]
	    - g_{11}(s,t)
      \\
      &=  0,
  \end{align*}
  which completes the proof.
\end{proof}

We now show that $G_{c,B}$ plays the role of a fundamental solution
for the operator $L^{c,B}$.  We let $C^{0,\infty}(\bb R\cross \bb
R^d)$ denotes the class of continuous functions that are infinitely
continuously differentiable with respect to the spacial variables.
\begin{lemma} \label{thm:3.6} 
  Let $c : \bb R \rightarrow \Smu$ be measurable and fix some
  $f \in C^\infty_K(\bb R^{1+d})$.
  Then $G_{c,B} f \in C^{0,\infty}(\bb R \cross \bb R^d)$ and for
  each multiindex $\alpha \in (\{0\} \cup \bb N)^d$ and fixed time
  $s \in \bb R$ there exists a constant $N = N(B, f, \alpha, s)$ such
  that
  \begin{equation} \label{eq:3.7}
    \|D^\alpha_x G_{c,B} f\|_{L^\infty([s, \infty) \cross \bb R^{d})}
    \leq N.
  \end{equation}
  Moreover, $D^\alpha_x G_{c,B} f$ admits the representation
  \begin{equation} \label{eq:3.8}
    D^\alpha_x G_{c,B} f(s,x)
    = \lim_{\varepsilon \rightarrow 0+}
      \int_{s+\varepsilon}^{s+1/\varepsilon}
      \int_{\bb R^d} D^\alpha_x p_{c,B}(s,x;t,y) f(t,y)
      \, \diff y \, \diff t.
  \end{equation}
  If we further assume that $c \in C^\infty(\bb R_+; \Smu)$,
  then $u = G_{c,B} f$ is a classical solution to the backwards
  equation $\partial_s u + L^{c,B} u = - f$ on $\bb
  R^{1+d}$.
\end{lemma}
\begin{proof}
  Set $g(s, x; t, y) = f(t,  e^{(t-s)B}x+y)$ and define
  \begin{align}
    \label{eq:3.9}
    u^\varepsilon(s,x)
    &\defeq \int_{s+\varepsilon}^{s+1/\varepsilon}
       \int_{\bb R^d} p_{c,B}(s,x;t,y) f(t, y) \, \diff y \, \diff t,
    \\
    \label{eq:3.10}
    v^\alpha(s,x)
    &\defeq \int_s^\infty \int_{\bb R^d}
      p_{c,B}(s,0;t,y) D^\alpha_x g(s,x; t,y) \, \diff y \, \diff t.
  \end{align}
  As $p_{c,B}(s,x;\,t,y) = p_{c,B}(s,0;\,t,y - e^{(t-s)B}x)$, we have
  \begin{equation} \label{eq:3.11}
    u^\varepsilon(s,x)
    = \int_{s+\varepsilon}^{s+1/\varepsilon}
       \int_{\bb R^{d}} p_{c,B}(s,0;t,y) g(s,x; t, y) \, \diff y \, \diff t.
  \end{equation}
  We then choose $T$ so that the support of $f$ is contained in
  the set $(-\infty, T] \cross \bb R^d$, and we observe that
  \begin{equation} \label{eq:3.12}
    |D^\beta_x g(s,x;t,y)|
    \leq \indd{t \leq T} e^{(t-s) |\beta| \|B\|}
	 \|D^\beta_x f\|_{L^\infty(\bb R^{1+d})},
  \end{equation}
  for any multiindex $\beta$.
  In particular, we may differentiate~\eqref{eq:3.11} repeatedly to obtain
  \begin{equation} \label{eq:3.13}
    D^\alpha_x u^\varepsilon(s,x)
    = \int_{s+\varepsilon}^{s+1/\varepsilon}
      \int_{\bb R^{d}} p_{c,B}(s,0;t,y) D^\alpha_x g(s,x; t, y)
		       \, \diff y \, \diff t.
  \end{equation}
  It follows from Lemma~\ref{thm:3.3} and dominated convergence that
  $D^\alpha_x u^\varepsilon$ is continuous.
  %
  %
  It then follows from~\eqref{eq:3.10}, \eqref{eq:3.12},
  and~\eqref{eq:3.13} that
  \begin{equation*}
    |v^\alpha(s,x) - D^\alpha_x u^\varepsilon(s,x)|
    \leq \varepsilon \, e^{\varepsilon |\alpha| \|B\|}
	 \|D^\alpha_x f\|_{L^\infty(\bb R^{1+d})},
  \end{equation*}
  when $s+1/\varepsilon > T$, so $D^\alpha_x u^\varepsilon$ converges
  to $v^\alpha$ uniformly on each set of the form $[u,\infty) \cross
  \bb R^{d}$ as $\varepsilon \rightarrow 0$.  As a result, we may
  conclude that $v^\alpha$ is continuous, $D^\alpha_x G_{c,B} f$
  exists, and $D^\alpha_x G_{c,B} f = v^\alpha$ for each multiindex
  $\alpha$.  In particular,~\eqref{eq:3.7} follows immediately
  from~\eqref{eq:3.12}.

  It follows from Lemma~\ref{thm:3.3}, Lemma~\ref{thm:3.4}, and
  dominated convergence that we may differentiate~\eqref{eq:3.9} to
  obtain
  \begin{align}
    D^\alpha_x u^\varepsilon(s,x)
      &= \int_{s+\varepsilon}^{s+1/\varepsilon}
	 \int_{\bb R^d} D^\alpha_x p_{c,B}(s,x;t,y) f(t, y)
		       \, \diff y \, \diff t.
  \end{align}
  In particular, we see that $D^\alpha_x G_{c,B} f$ admits the
  representation~\eqref{eq:3.8}.

  When $c \in C^\infty(\bb R_+; \Smu)$, another domination argument
  shows that we have $\partial_s u^\varepsilon(s,x) = -
  w^\varepsilon_0(s,x) + w^\varepsilon_1(s,x) + w^\varepsilon_2(s,x)$,
  where
  \begin{align*}
    w^\varepsilon_0(s,x)
    &\defeq \int_{\bb R^{d}} p_{c,B}(s,x;s+\varepsilon,y)
			  f(s+\varepsilon, y) \, \diff y,
    \\
    w^\varepsilon_1(s,x)
    &\defeq \int_{s+\varepsilon}^{s+1/\varepsilon} \int_{\bb R^{d}}
	       \partial_s p_{c,B}(s,x;t,y)
	       f(t, y)	\, \diff t \, \diff y,
    \\
    w^\varepsilon_2(s,x)
    &\defeq \int_{\bb R^{d}} p_{c,B}(s,x;s+1/\varepsilon,y)
			  f(s+1/\varepsilon, y) \, \diff y.
  \end{align*}
  As $f$ is uniformly continuous, $w^\varepsilon_0 \rightarrow f$
  uniformly on $\bb R^{1+d}$.  As $f$ has compact support,
  $w^\varepsilon_2 \rightarrow 0$ uniformly on sets of the form
  $[u,\infty)\cross \bb R^d$.  Lemma~\ref{thm:3.4} asserts that
  $\partial_s p_{c,B}(s,x;t,y) + L^{c,B}_{\xbar}{p_{c,B}}(s,x;t,y) =
  0$ when $s < t$, so
  \begin{equation*}
    w^\varepsilon_1(s,x)
    = - \int_{s+\varepsilon}^{s+1/\varepsilon} \int_{\bb R^{d}}
		L^{c,B}_{\xbar}{p_{c,B}}(s,x;t,y)
		f(t, y)	 \, \diff t \, \diff y.
  \end{equation*}
  We have already shown that $D^\alpha_x u^\varepsilon$ converges
  uniformly to $D^\alpha_x G_{c,B} f$ on sets of the form $[u,\infty)
  \cross \bb R^{d}$ for each multiindex $\alpha$.  The coefficients of
  the operator $L^{c,B}$ are locally bounded, so $w^\varepsilon_1$
  converges to $L^{c,B} G_{c,B} f$ uniformly on compact sets.  Putting
  this all together, we see that $\dx{s} u^\varepsilon$ converges to
  $-f - L^{c,B} G_{c,B} f$ and this convergence is uniform on compact
  sets, so $\partial_s G_{c,B} f$ exists and
  equals $-f - L^{c,B} G_{c,B} f$.
\end{proof}

\section{An $L^p$-estimate} \label{sec:4}

The following theorem is main result of this section.

\begin{theorem} \label{thm:4.1} %
  Let $c : \bb R \rightarrow \Smu$ be measurable, $i,j \in \{1, \dots,
  d_0\}$, $p \in (1, \infty)$, and $f \in C^\infty_K(\bb R^{1+d})$.
  Then $\|\partial_{ij} G_{c,B} f \|_{L^p(\bb R^{1+d})} \leq
  N(B,\mu,p) \, \| f \|_{L^p(\bb R^{1+d})}$.
\end{theorem}

We will obtain this estimate by studying the singular integral
representation of $\partial_{ij} G_{c,B}$.  The approach that we
follow is essentially a mixture of techniques from Section XIII.5
of~\cite{stein1993harm} and Section A.2 of~\cite{stroock1979mdp}.  To
reduce the notational burden in this section, we will collect all of
the information that we need to specify a kernel in a single tuple.
Let $B(\bb R; \Smu)$ denote the set of bounded, measurable functions
from $\bb R$ to $\Smu$ and set $\scr A(\mu) = B(\bb R; \Smu) \cross
\{1,\dots,d_0\}^2 \cross \{0,1\}$.  Given a tuple $\alpha =
(c,k,\ell,m) \in \scr A(\mu)$, we define the singular kernel
\begin{equation*}
  h_{\alpha}(s,x;t,y)
  \defeq \indd{m=0} \dfxx{p_{c,B}}{x^k}{x^\ell}(s,x;t,y)
	 + \indd{m=1} \dfxx{p_{c,B}}{x^k}{x^\ell}(t,y; s,x).
\end{equation*}
We will also need the truncated kernels
\begin{equation*}
  h^{i}_{\alpha}(s,x;t,y)
  \defeq \ind{(1,4]}(|t-s| / 4^{i}) \, h_{\alpha}(s,x;t,y),
  \qquad i \in \bb Z,
\end{equation*}
and the operators
\begin{alignat*}{2}
  H^{i}_{\alpha} f(\xbar)
  &\defeq \int_{\bb R^{1+d}} h^i_{\alpha}(\xbar; \ybar)
		 f(\ybar) \, \diff \ybar,
  &\qquad &i \in \bb Z,
  \\
  K^j_\alpha f(\xbar) &\defeq \sum_{i=-j}^{j} H^i_\alpha f(\xbar),
  &\qquad &j \in \bb N.
\end{alignat*}
Each kernel $h^i_\alpha$ is bounded, so these operators are defined in
a pointwise sense for all $f \in C_K(\bb R^{1+d})$.  The main
task in this section is to show that the collection of operators $\scr
K = \{ K^{j}_\alpha : \alpha \in \scr A(\mu), j \in \bb N\}$ is
uniformly bounded with respect to the $L^p$-operator norm for each $p
\in (1,\infty)$.  Once this is done, Theorem~\ref{thm:4.1} follows
easily from Fatou's Lemma.

We will first obtain a uniform bound with respect to the
$L^2$-operator norm using the Cotlar-Stein Almost Orthogonality Lemma,
which we now recall for the reader's convenience.  One proof of this
lemma may be found in Section~VII.2 of \cite{stein1993harm}.
\begin{lemma}[Cotlar-Stein Lemma]
  \label{thm:4.2}
  Let $\{T_i\}$ be a sequence of bounded operators on some $L^2$-space
  and let $\{c_i\}_{i=-\infty}^\infty$ be a sequence of positive
  constants with $N = \sum_{i=-\infty}^\infty c_i < \infty$.  If $\|
  T^*_i T_j \| \leq c^2_{i-j}$ and $\| T_i T^*_j \| \leq c^2_{i-j}$, then
  $\|\sum_{i=-n}^n T_i\| \leq N$ for all $n \geq 1$.
\end{lemma}

Once we have a uniform bound in the $L^2$-operator norm, we will check
that the kernels associated with the operators $K^j_\alpha$ satisfy an
integrable Hormander condition which is adapted to our geometric
setting.  This will allow us to obtain a uniform bound with respect to
the $L^p$-operator norm for $p \in (1,2)$ using the following theorem.
\begin{theorem} \label{thm:4.3} %
  Let $p \in (1,2)$, let $k$ be a bounded, measurable function, and
  set
  \begin{equation} \label{eq:4.1}
    K f(\xbar) = \int_{\bb R^{1+d}} k(\xbar; \ybar) \, f(\ybar) \,
		 \diff \ybar,
    \qquad f \in C^\infty_K(\bb R^{1+d}).
  \end{equation}
  Suppose that there exists a constant $N_1$ such that
  \begin{equation} \label{eq:4.2}
    \begin{gathered}
      \int_{\rho(\zbar\inv\op\xbar) \geq N_1 \rho(\zbar\inv\op\ybar)}
	|k(\xbar; \ybar) - k(\xbar; \zbar)| \, \diff \xbar
	\leq N_1,
    \end{gathered}
  \end{equation}
  and a constant $N_2$ such that $\|Kf\|_{L^2} \leq N_2
  \|f\|_{L^2}$ when $f \in C^\infty_K(\bb R^{1+d})$.
  Then there exists a constant $N = N(N_1,
  N_2, p)$ such that $\| H f \|_{L^p} \leq N
  \|f\|_{L^p}$ when $f \in C^\infty_K(\bb R^{1+d})$.
\end{theorem}
The reader may consult Theorem 3 in Section I.5 of
\cite{stein1993harm} for a proof of Theorem~\ref{thm:4.3} under
weaker hypotheses.  As the collection of operators $\scr K$ is closed
with respect to taking the (formal) adjoint, uniform bounds with
respect to the $L^p$-operator norm for $p \in (2,\infty)$ will then
follow from duality.	It is this last case that we will need in
Section~\ref{sec:5}.

We will begin the process by listing the translation and scaling
properties of the kernels $h^i_\alpha$.  These properties follow
easily from the properties of $p_{c,B}$ given in Lemma~\ref{thm:3.1},
the explicit formula for $\dfxx{p_{c,B}}{x^k}{x^\ell}(s,x;t,y)$ given
in the proof of Lemma~\ref{thm:3.4}, and the fact that $B(\bb R;
\Smu)$ is closed with respect to translation and dilation, so we
leave the verification of this lemma to the reader.  We remind the
reader that $\dbar \geq d + 2$ denotes the homogeneous dimension of
the group $(\bb R^{1+d}, \,\op\,, \dilbar)$ as defined in
Section~\ref{sec:2}.  We also point out that the exponent in the
dilation law~\eqref{eq:4.4} becomes less favorable if we attempt to
differentiate $p_{c,B}$ with respect to $x^i$ with $i > d_0$.  This
explains to a large extent why we must wait until we get to the
probabilistic level to make any changes to the drift.
\begin{lemma} \label{thm:4.4} %
  Let $\alpha \in \scr A(\mu)$ and fix some
  $\zbar = (u,z) \in \bb R^{1+d}$ and $\lambda > 0$.  Then we may find
  $\beta, \gamma \in \scr A(\mu)$ such that
  \begin{align}
    \label{eq:4.3}
    h_{\alpha}(\zbar \op \xbar; \;\zbar \op \ybar)
    &= h_{\beta}(\xbar; \; \ybar),
    \\
    \label{eq:4.4}
    h_{\alpha}(\dil_\lambda \xbar; \;\dil_\lambda \ybar)
    &= \lambda^{-\dbar} \, h_{\gamma}(\xbar; \;\ybar).
  \end{align}
\end{lemma}
\begin{remark}
  For each $\alpha \in \scr A(\mu)$, we may find $\beta \in \scr A(\mu)$ such
  that
  \begin{equation*}
    h^i_{\alpha}(\xbar; \ybar)
    = 2^{-i\dbar} h^0_{\beta}(\dilbar_{2^{-i}} \xbar; \; \dilbar_{2^{-i}} \ybar).
  \end{equation*}
  But the Jacobian determinant of the map $\xbar \mapsto
  \dilbar_\lambda \xbar$ is $\lambda^\dbar$, so
  \begin{align*}
    \int_{\bb R^{1+d}} h^i_{\alpha}(\xbar; \ybar)  f(\xbar) \, \diff \xbar
    = \int_{\bb R^{1+d}} h^0_{\beta}(\xbar; \dilbar_{2^{-i}} \ybar)
      f(\dilbar_{2^i} \xbar) \, \diff \xbar,
  \end{align*}
  when either integral is well-defined.
  In the remainder of this section, when we say ``by dilation'', we are
  making use of (minor variations on) this observation.
\end{remark}

\begin{lemma} \label{thm:4.6} %
  There exists an function $\hhat_\mu : \bb R^{d} \rightarrow \bb R_+$
  such that
  \begin{align}
    \label{eq:4.5}
    | h_{\alpha}(s,x; \ybar) | &\leq \hhat_\mu(x),
    \\
    \label{eq:4.6}
    | h_{\alpha}(s,x; \ybar) - h_{\alpha}(s,x; 0) |
    &\leq |\ybar| \, \hhat_\mu(x),
  \end{align}
  for all $|s| \in [1/2, 5]$, $|\ybar| \leq 1/4$, and $\alpha \in \scr
  A(\mu)$.  Moreover, $\hhat_\mu$ may be chosen such that $\int_{\bb
    R^d} |f(x)| \, \hhat_\mu(x) \, \diff x < \infty$ for every
  function $f$ of polynomial growth.
\end{lemma}
\begin{proof}
  Let $\xbar = (s,x)$, $\ybar = (t,y)$, and set $k^{ij}_c(s,x;t,y) \defeq
  \dxx{x^i}{x^j}p_{c,B}(s,x;t,y)$ for $0 \leq i,j \leq d_0$ and $c \in
  C^\infty(\bb R; \Smu)$.  Using Lemma~\ref{thm:3.4}, we may find a
  polynomial $P$ in four variables with positive coefficients
  such that
  \begin{align*}
    \lefteqn{
	     |k^{ij}_c(s,x;t,y)| +
	     |D_{\xbar} k^{ij}_c(s,x; t,y)|
	     + |D_{\ybar} k^{ij}_c(s,x; t,y)|
    }
    \qquad \qquad \qquad \qquad \qquad \qquad \\
    &\leq P(t-s, 1/(t-s), |x|, |y|) \, p_{c,B}(s,x;t,y),
  \end{align*}
  when $t > s$.
  Next we use Lemma~\ref{thm:3.3} to produce constants $N = N(B,\mu)$
  and $\varepsilon = \varepsilon(B,\mu) > 0$ such that $p_{c,B}(s,x;t,y)
  \leq N \exp\{N |x| |y| - \varepsilon(|x|^2 + |y|^2)\}$ when $t-s \in
  [1/4,6]$, and we set
  \begin{align*}
    \hhat_\mu(x)
    &= N \{ P(6,4,x,1/4) + P(6,4,1/4,x) \} \,
       \exp\{N |x| / 4 - \varepsilon |x|^2\}.
  \end{align*}
  As $P$ is a polynomial, $f \, \hhat_\mu$ is integrable when $f$ has
  polynomial growth.

  We now check that~\eqref{eq:4.5} and~\eqref{eq:4.6} holds.  Let $\alpha
  = (c,i,j,0) \in \scr A(\mu)$ with $c \in C^\infty(\bb R; \Smu)$ and
  $1 \leq i,j \leq d_0$.  If $s \in [-5, -1/2]$ and $|\ybar| \leq
  1/4$, then $t-s \in [1/4, 6]$,
  \begin{align*}
    |h_\alpha(s,x;\ybar)|
    &= |k^{ij}_c(s,x;\ybar)|
    \\
    &\leq P(6,4,|x|,1/4) \sup_{|\zbar| \leq 1/4} p_{c,B}(s,x; \zbar)
    \\
    &\leq P(6,4,|x|,1/4) \, N \exp\{N |x| / 4 - \varepsilon |x|^2\}
    \leq \hhat_\mu(x),
  \end{align*}
  and
  \begin{align*}
    |h_\alpha(s,x;\ybar) - h_\alpha(s,x;0)|
    &= |k^{ij}_c(s,x;\ybar) - k^{ij}_c(s,x;0)|
    \\
    &\leq |\ybar| \sup_{|\zbar| \leq 1/4} |D_{\ybar} k^{ij}_c(s,x;\zbar)|
    \leq |\ybar| \, \hhat_\mu(x).
  \end{align*}
  Of course, if $s \in [1/2, 5]$ and $|\ybar| \leq 1/4$, then
  $h_\alpha(s,x;\ybar) - h_\alpha(s,x;0) = 0$, and the inequalities
  holds trivially.

  If we choose any $\beta = (c,i,j,1) \in \scr A(\mu)$ with $c \in
  C^\infty(\bb R; \Smu)$ and $1 \leq i,j \leq d_0$, then
  $h_\alpha(s,x;\ybar) - h_\alpha(s,x;0) = k^{ij}_c(\ybar; s,x) -
  k^{ij}_c(0; s,x)$ and the estimate follows in the same way as the
  previous case.  Finally, to handle the case where $\gamma =
  (c,i,j,k) \in \scr A(\mu)$, but $c$ is only measurable, we choose
  $c_n \in C^\infty(\bb R; \Smu)$ with $\int_I \|c(s) - c_n(s)\| \,
  \diff s \rightarrow 0$ for each compact interval $I$.  Then
  $h_{(c_n,i,j,k)} \rightarrow h_\gamma$ pointwise, and~\eqref{eq:4.5}
  and~\eqref{eq:4.6} hold for each $h_{(c_n,i,j,k)}$, so they also
  hold for $h_\gamma$.
  %
  %
\end{proof}

We will make use the following easy corollary a couple of times.
\begin{corollary} \label{thm:4.7}
  Let $\alpha \in \scr A(\mu)$ and $\ybar \in \bb R^{1+d}$.  Then
  there exists a constant $N = N(B,\mu)$ such that $\int_{\bb
    R^{1+d}} |h^0_\alpha(\xbar; \ybar)| \, \diff \xbar \leq N$.
\end{corollary}
\begin{proof}
  Let $\hhat_\mu$ denote function defined in the previous lemma.
  Making use of left-translation and change of variable, we may find
  $\beta \in \scr A(\mu)$ such that
  \begin{align*}
    \int_{\bb R^{1+d}} |h^0_\alpha(\xbar; \ybar)| \, \diff \xbar
    &= \int_{\bb R^{1+d}} |h^0_\beta(\xbar; 0)| \, \diff \xbar
    \leq \int_{\bb R} \int_{\bb R^d} \ind{[1,4]}(|s|) \,
	  \hhat_\mu(x) \, \diff x \, \diff s,
  \end{align*}
  and this last integral is finite.
\end{proof}

\begin{lemma} \label{thm:4.8} %
  Let $\alpha = (c,k,\ell,m) \in \scr A(\mu)$ and let $p \in
  (1,\infty)$.	Then there exists a unique bounded linear operators
  $T: L^p(\bb R^{1+d}) \rightarrow L^p(\bb R^{1+d})$ such that $T$
  agrees with $H^i_\alpha$ on $C^\infty_K(\bb R^{1+d})$.  Moreover, if
  we set $q = p / (p-1)$ and $\alpha^* = (c,k,\ell,1-m) \in \scr A(\mu)$,
  then $T^* : L^q(\bb R^{1+d}) \rightarrow L^q(\bb R^{1+d})$ is the unique
  bounded operator that agrees with $H^i_{\alpha^*}$ on $C^\infty_K(\bb
  R^{1+d})$.
\end{lemma}
\begin{proof}
  By dilation, it is enough to show that the lemma holds for
  $H^0_{\alpha}$.  Let $f,g \in C^\infty_K(\bb R^{1+d})$, set $E =
  \bb R^{1+d}$, and choose $N = N(B,\mu)$ as in
  Corollary~\ref{thm:4.7} so that $\int_{\bb R^{1+d}}
  |h^0_\alpha(\xbar; \ybar)| \, \diff \xbar \leq N$ for all $\alpha
  \in \scr A$ and $\ybar \in \bb R^{1+d}$. It then follows from
  Tonelli's Theorem and Young's Inequality that
  \begin{align*}
    \int_E |H^0_\alpha f(\xbar) g(\xbar)| \, \diff \xbar
    &\leq \int_E \int_E
	  |h^0_\alpha(\xbar; \ybar)| \, |f(\ybar)|^p / p
	  \, \diff \xbar \, \diff\ybar
    \\ &\qquad
	+ \int_E \int_E
	  |h^0_{\alpha^*}(\xbar; \ybar)| \, |g(\ybar)|^q / q
	  \, \diff \xbar \, \diff\ybar
    \\
    &\leq N \, \bigl(\|f\|_{L^p} / p + \|g\|_{L^p} / q\bigr).
  \end{align*}
  Taking the supremum over $f$ and $g$ with $\|f\|_{L^p} \leq 1$ and
  $\|g\|_{L^q} \leq 1$, we see that $\|H^0_\alpha f\|_{L^p} \leq N
  \|f\|_{L^p}$, so $H^0_\alpha$ extends uniquely to a bounded operator
  $T$ on $L^p(\bb R^{1+d})$.  Moreover, if $f,g \in C^\infty_K(\bb
  R^{1+d})$, then
  \begin{align*}
    \int_E f (\xbar) T^* g(\xbar) \, \diff \xbar
    &= \int_E T f (\xbar) g(\xbar) \, \diff \xbar
    \\
    &= \int_{E} H^i_\alpha f(\xbar) g(\xbar) \, \diff \xbar
    = \int_{E} f(\xbar) H^i_{\alpha^*} g(\xbar) \, \diff \xbar,
  \end{align*}
  where the use of Fubini's Theorem in the last equality is justified
  by the previous inequality.  By varying $f$, we may conclude that
  $H^i_{\alpha^*} g$ is a version of $T^* g$.
\end{proof}

\begin{lemma} \label{thm:4.9}
  There exists a constant $N = N(B, \mu)$ such that
  \begin{equation} \label{eq:4.7}
    \int_{\bb R^{1+d}} | h^i_{\alpha}(\xbar; \ybar)
	     - h^i_{\alpha}(\xbar; \zbar) |
	   \; \diff \xbar
	   \leq N 2^{-i} \rho(\zbar\inv\op\ybar).
  \end{equation}
  for all $\alpha \in \scr A(\mu)$ and $i \in \bb Z$.
\end{lemma}
\begin{proof}
  By left-translation and dilation, it is
  enough to produce a constant $N = N(\mu)$ such that
  \begin{equation} \label{eq:4.8}
      \int_{\bb R^{1+d}} | h^0_{\alpha}(\xbar; \ybar) - h^0_{\alpha}(\xbar; 0) |
	     \; \diff \xbar
      \leq N  \rho(\ybar).
  \end{equation}
  Corollary~\ref{thm:4.7} asserts that we may choose a constant $N_1 =
  N_1(B,\mu)$ such that
  \begin{equation} \label{eq:4.9}
    \int_{\bb R^{1+d}} | h^0_{\alpha}(\xbar; \ybar) - h^0_{\alpha}(\xbar; 0) |
	   \; \diff \xbar
    \leq N_1.
  \end{equation}
  This is a useful bound when $\rho(\ybar)$ is large.

  We now set $(t,y) = \ybar$ and consider the case where $|\ybar| \leq
  1/4$.	 Using Lemma~\ref{thm:4.6}, we may choose a integrable function
  $\hhat_\mu : \bb R^{d} \rightarrow \bb R_+$ such that $|
  h_{\beta}(s,x; \ybar) | \leq \hhat_\mu(x)$ and $| h_{\beta}(s,x;
  \ybar) - h_{\beta}(s,x; 0) | \leq |\ybar| \, \hhat_\mu(x)$ when
  $\beta \in \scr A(\mu)$, $|\ybar| \leq 1/4$, and $|s| \in [1/2,5]$.
  Set $N_2 = N_2(B,\,u) = \int_{\bb R^d} \hhat_\mu(x) \, \diff x$.
  When $|\ybar| \leq 1/4$, we have
  $\bigl|\ind{[1,4]}(|t-s|) - \ind{[1,4]}(|s|)\bigr| \leq
  \ind{[1/2,5]}(|s|)$ and
  \begin{align*}
    \lefteqn{
    \int_{\bb R^{1+d}}
	  | h^0_{\alpha}(\xbar; \ybar) - h^0_{\alpha}(\xbar; 0) |
	   \; \diff \xbar} \qquad \qquad \qquad \\
    &\leq \int_{\bb R} \int_{\bb R^{d}} \bigl|\ind{[1,4]}(|t-s|)
					      - \ind{[1,4]}(|s|)\bigr|
			       \; | h_{\alpha}(s,x; \ybar) |
			       \, \diff x \, \diff s
	  \\ & \qquad \qquad \qquad
	  + \int_{\bb R} \int_{\bb R^{d}} \ind{[1,4]}(|s|) \,
			      \bigl| h_{\alpha}(\xbar; \ybar)
				     - h_{\alpha}(\xbar; 0) \bigr|
			      \, \diff x \, \diff s
    \\
    &\leq \int_{\bb R} \bigl|\ind{[1,4]}(|t-s|)
			     - \ind{[1,4]}(|s|)\bigr| \, \diff s
	  \int_{\bb R^{d}} \hhat_\mu(x) \, \diff x
	  \\ & \qquad \qquad \qquad
	  + |\ybar| \int_{\bb R} \ind{[1,4]}(|s|) \, \diff s
		    \int_{\bb R^{d}} \hhat_\mu(x) \, \diff x
    \\
    &\leq 4|t| N_1 + 3 |\ybar| N_1.
  \end{align*}
  As $\rho(\ybar) \geq |\ybar|$ when $|\ybar| \leq 1$, we may
  conclude that
  \begin{equation} \label{eq:4.10}
    \int_{\bb R^{1+d}}
	  | h^0_{\alpha}(\xbar; \ybar) - h^0_{\alpha}(\xbar; 0) |
	   \; \diff \xbar
    \leq 7 N_1 \rho(\ybar),
    \qquad \text{when $\rho(\ybar) \leq 1/4$.}
  \end{equation}
  We may then produce a constant such that~\eqref{eq:4.8}
  holds by using~\eqref{eq:4.10} when $\rho(\ybar)$ is small and
  using~\eqref{eq:4.9} when $\rho(\ybar)$ is large.
\end{proof}
We now produce the desired bound with respect to the $L^2$-operator norm.
\begin{lemma} \label{thm:4.10} %
  There exists a constant $N=N(B,\mu)$ with $\| K^{j}_{\alpha} f
  \|_{L^2} \leq N \| f \|_{L^2}$ for all $\alpha \in \scr A(\mu)$, $j
  \in \bb N$ and $f \in C^\infty_K(\bb R^{1+d})$.
\end{lemma}
\begin{proof}
  Set $E = \bb R^{1+d}$, and let $T^i_\alpha : L^2(E) \rightarrow
  L^2(E)$ denote the unique, bounded operator that agrees with
  $H^i_\alpha$ on $C^\infty_K(E)$.  We will show that the collection
  of operators $\mathcal T = \bigl\{ \sum_{i=-n}^n T^i_\alpha \, : \,
  \alpha \in \scr A(\mu), n \in \bb N \bigr\}$ is uniformly bounded.
  By the Cotlar-Stein Lemma, it is enough to produce a constant $N$
  such that $\|(T^{i}_\alpha)^* T^{j}_\alpha\| \leq N 2^{-|i-j|}$ and
  $\|T^{i}_\alpha (T^{j}_\alpha)^*\| \leq N 2^{-|i-j|}$ for all
  $\alpha \in \scr A(\mu)$ and $i,j \in \bb Z$.  But Lemma
  \ref{thm:4.8} asserts that the class $\mathcal T$ is closed with
  respect to taking adjoints, so it sufficient to show that the first
  of these inequalities holds.

  We will, in fact, produce a constant $N = N(B, \mu)$ such that
  \begin{equation} \label{eq:4.11}
    \int_{E^2} |h^i_\alpha(\xbar;\ybar) h^j_\alpha(\xbar;\zbar)|
	       \, \diff \xbar \, \diff \zbar \leq N 2^{-|i-j|},
    \qquad \text{for all $\ybar \in E$.}
  \end{equation}
  To see that this is sufficient to prove the theorem, assume
  that~\eqref{eq:4.11} holds and choose any $f, g \in C^\infty_K(\bb
  R^{1+d})$.  We then have
  \begin{align*}
    \bigl\langle (T^{i}_\alpha)^* T^j_\alpha f, \, g \bigr\rangle
    &= \langle T^j_\alpha f, \, T^i_\alpha g \rangle
    \\
    &= \int_{E^3}
	 h^i_\alpha(\xbar;\ybar) h^j_\alpha(\xbar;\zbar) f(\ybar) g(\zbar)
	 \, \diff \xbar \, \diff \ybar \, \diff \zbar
    \\
    &\leq \frac{1}{2} \int_{E} \left\{ \int_{E^2}
	  |h^i_\alpha(\xbar;\ybar) h^j_\alpha(\xbar;\zbar)|
	  \, \diff \xbar \, \diff \zbar \right\} f^2(\ybar)
	  \, \diff \ybar
    \\
    &\qquad \qquad
       + \frac{1}{2} \int_{E}  \left\{ \int_{E^2}
	 |h^i_\alpha(\xbar;\ybar) h^j_\alpha(\xbar;\zbar)|
	 \, \diff \xbar \, \diff \ybar \right\}
	 g^2(\zbar) \, \diff \zbar
    \\
    &\leq \half N 2^{-|i-j|} \bigl(\|f\|_{L^2} +  \|g\|_{L^2}\bigr),
  \end{align*}
  where we have used Fubini's Theorem and Young's inequality.  If we
  then take the supremum over $f$ and $g$ with $\|f\|_{L^2} \leq 1$
  and $\|g\|_{L^2} \leq 1$, then we see that $\|(T^{i}_\alpha)^*
  T^j_\alpha \| \leq N 2^{-|i-j|}$ and the theorem follows.

  We now show that~\eqref{eq:4.11} holds.  Using
  Corollary~\ref{thm:4.7}, we find a constant $N_2 = N_2(B,\mu)$ such
  that $\int_{\bb R^{1+d}} |h^0_\alpha(\xbar; 0)| \, \rho(\xbar) \,
  \diff \xbar \leq N_2$ for all $\alpha \in \scr A(\mu)$.  By
  left-translation and dilation, we may choose $\beta, \gamma \in \scr
  A(\mu)$ such that
  \begin{equation*}
    \int_E h^i_{\alpha}(\xbar; \ybar) \; h^j_{\alpha}(\xbar; \zbar)  \,
	     \diff \xbar
    = \int_E h^0_{\beta}(\xbar; 0) \;
       h^j_{\gamma}(\dilbar_{2^i} \xbar;\; \ybar\inv\op \zbar) \, \diff \xbar.
  \end{equation*}
  After a change of variable and an application of the cancellation
  property given in Lemma~\ref{thm:3.5}, we have
  \begin{align*}
    \lefteqn{
    \int_E
    \int_E h^0_{\beta}(\xbar; 0)
		       h^j_{\gamma}(\dilbar_{2^i} \xbar;\;
                                    \ybar\inv \op \zbar)
	       \, \diff \xbar \, \diff \zbar} \qquad\qquad \\
    &= \int_E h^0_{\beta}(\xbar; 0) \,
	  \left\{ \int_E
		  h^j_{\gamma}(\dilbar_{2^i} \xbar;\; \zbar)
		  - h^j_{\gamma}(0;\; \zbar)
		  \, \diff \zbar \right\} \, \diff \xbar.
  \end{align*}
  Letting $N_3 = N_3(B, \mu)$ denote the constant obtained in
  Lemma~\ref{thm:4.9}, we see that
  \begin{align*}
    \lefteqn{
    \int_{E^2} |h^i_{\alpha}(\xbar; \ybar) h^j_{\alpha}(\xbar; \zbar)|
	       \, \diff \xbar \, \diff \zbar
    } \qquad \qquad \\
    &\leq \int_{E} |h^0_{\beta}(\xbar; 0)| \,
	  \left\{ \int_E
		  | h^j_\gamma(\dilbar_{2^i} \xbar; \zbar)
		  - h^j_\gamma(0; \zbar) |
		  \, \diff \zbar \right\} \, \diff \xbar
    \\
    &\leq 2^{-j} N_3 \int_{E} |h^0_{\beta}(\xbar; 0)| \,
		     \rho(\dilbar_{2^i} \xbar) \, \diff \xbar
    \\
    &= 2^{i-j} N_3 \int_{E} |h^0_{\beta}(\xbar; 0)| \, \rho(\xbar) \, \diff \xbar
    \\
    &\leq N_2 \, N_3 \, 2^{i-j}.
  \end{align*}

  To handle the case where $i > j$, we choose new $\beta, \gamma \in
  \scr A(\mu)$ such that
  \begin{equation*}
    \int_E h^i_{\alpha}(\xbar; \ybar) \, h^j_{\alpha}(\xbar; \zbar)  \,
	     \diff \xbar
    = \int_E h^i_{\beta}(\dilbar_{2^j} \xbar; \zbar\inv\op \ybar)
      \, h^0_{\gamma}(\xbar;\; 0) \, \diff \xbar.
  \end{equation*}
  The matrix $B$ is strictly lower triangular, so $\det e^{-sB} = \pm
  1$.  In particular, if we fix some $\ybar = (t,y) \in E$, then we see
  that the absolute value of the Jacobian determinant of the map
  $\zbar = (u,z) \mapsto \zbar\inv \op \ybar = (t-u, y -
  e^{(t-u)B}z)$ is one.	 This means that
  \begin{align*}
    \int_E
      h^i_{\beta}(\dilbar_{2^i} \xbar; \zbar\inv\op \ybar) \,
      \diff \zbar
    &= \int_E
       h^i_{\beta}(\dilbar_{2^i} \xbar; \zbar) \,
	\diff \zbar,
  \end{align*}
  for each fixed $\xbar$.  Arguing as in the previous case, we see that
  \begin{align*}
    \int_{E^2} |h^i_{\alpha}(\xbar; \ybar) h^j_{\alpha}(\xbar; \zbar)|
	       \, \diff \xbar \, \diff \zbar
    &\leq \int_E
	  \left\{ \int_E
		  | h^i_\beta(\dilbar_{2^j} \xbar; \zbar)
		  - h^i_\beta(0; \zbar) |
		  \, \diff \zbar \right\}
		  h^0_{\gamma}(\xbar;\; 0) \, \, \diff \xbar
    \\
    &\leq N_2 \, N_3 \, 2^{j-i}.
  \end{align*}
  We have now shown that~\eqref{eq:4.11} holds, so the proof is
  complete.
\end{proof}

\begin{lemma} \label{thm:4.11} %
  Let $p \in (1,2)$.  Then there exists a constant $N= N(B,\mu,p)$
  such that $\|K^j_\alpha f\|_{L^p} \leq N \|f\|_{L^p}$ for all
  $\alpha \in \scr A(\mu)$, $j \in \bb N$, and $f \in C^\infty_K(\bb
  R^{1+d})$.
\end{lemma}
\begin{proof}
  First we observe that it is enough to produce a constant $N_1 =
  N_1(B,\mu)$ such that
  \begin{align}
    \label{eq:4.12}
    \sum_{i=-\infty}^\infty
    \int_{\rho(\zbar\inv\op\xbar) \geq N_1 \rho(\zbar\inv\op\ybar)}
	|h^{i}_{\alpha}(\xbar; \ybar) - h^{i}_{\alpha}(\xbar; \zbar)|
	\, \diff \xbar &\leq N_1.
  \end{align}
  To see this, suppose that~\eqref{eq:4.12} holds and set
  $k^j_\alpha(\xbar; \ybar) = \sum_{i=-j}^{j} h^i_\alpha(\xbar;
  \ybar)$.  Then $k^j_\alpha$ is bounded, $K^j_\alpha f(\xbar) =
  \int_{\bb R^{1+d}} k^j_\alpha(\xbar;\ybar) f(\ybar) \, \diff \ybar$
  when $f \in C^\infty_K(E)$, and
  \begin{equation*}
    \int_{\rho(\zbar\inv\op\xbar) \geq N_1 \rho(\zbar\inv\op\ybar)}
	|k^{j}_{\alpha}(\xbar; \ybar) - k^{j}_{\alpha}(\xbar; \zbar)|
	\, \diff \xbar \leq N_1.
  \end{equation*}
  Lemma~\ref{thm:4.10} asserts that we may choose a constant $N_2 =
  N_2(B,\mu)$ such that $\| K^{j}_{\alpha} f \|_{L^2} \leq N \| f
  \|_{L^2}$ for all $f \in C^\infty_K(\bb R^{1+d})$.  We may then
  invoke Theorem~\ref{thm:4.3} to produce a constant $N = N(N_1, N_2,
  p)$ such that $\| K^{j}_{\alpha} f \|_{L^p} \leq N \| f \|_{L^p}$
  for all $f \in C^\infty_K(\bb R^{1+d})$.  The constants $N_1$ and
  $N_2$ only depend upon $B$ and $\mu$, so the constant $N$ only
  depends upon $B$, $\mu$, and $p$.

  Rather than prove~\eqref{eq:4.12} directly, we will instead produce a
  constant $N_1 = N_1(B,\mu)$, such that
  \begin{align}
    \label{eq:4.13}
    \sum_{i=-\infty}^\infty
    \int_{\rho(\xbar) \geq N_1 / 2}
	|h^{i}_{\alpha}(\xbar; \ybar) - h^{i}_{\alpha}(\xbar; 0)|
	\, \diff \xbar &\leq N_1,
    \qquad \text{when $\rho(\ybar) \leq 1$.}
  \end{align}
  It is easy to check that~\eqref{eq:4.12} follows from~\eqref{eq:4.13}
  after a left-translation that moves $\zbar$ to zero and a dilation that
  puts $\rho(\ybar) \in (1/2,1]$.

  We will show that~\eqref{eq:4.13} holds by handling the terms
  where $i \geq 0$ and $i < 0$ separately.  To handle the terms where
  $i \geq 0$, we invoke Lemma~\ref{thm:4.9} to produce a constant
  $N_2 = N_2(B,\mu)$ such that
  \begin{equation*}
    \int_E |h^{i}_{\alpha}(\xbar; \ybar)
	    - h^{i}_{\alpha}(\xbar; 0)| \, \diff \xbar
    \leq 2^{-i} N_2 \rho(\ybar).
  \end{equation*}
  In particular, we have
  \begin{equation} \label{eq:4.14}
    \sum_{i=0}^\infty \int_E
      |h^{i}_{\alpha}(\xbar; \ybar) - h^{i}_{\alpha}(\xbar; 0)| \, \diff \xbar
      \leq 2 N_2,
      \qquad \text{ when $\rho(\ybar) \leq 1$.}
  \end{equation}

  We now handle the terms where $i < 0$.  The map $(\xbar, \zbar)
  \mapsto \rho(\zbar\op\xbar)$ is continuous and bounded on the
  compact set $\{ (\xbar,\zbar) \in \bb R^{1+d} \cross \bb R^{1+d} :
  \rho(\xbar) \leq 1, \, \rho(\zbar) \leq 1\}$, so we may choose $N_3
  = N_3(B)$ so large that $\rho(\xbar) \geq 1$ when
  $\rho(\zbar\op\xbar) \geq N_3$ and $\rho(\zbar) \leq 1$.  We then
  choose $\beta, \gamma \in \scr A(\mu)$ such that
  \begin{align*}
    \int_{\rho(\xbar) \geq N_3} |h^{i}_{\alpha}(\xbar; \zbar)| \, \diff \zbar
    &= \int_{\rho(\zbar\op\xbar) \geq N_3}
	   |h^{i}_{\beta}(\xbar; 0)| \, \diff \xbar
    \\
    &\leq \int_{\rho(\xbar) \geq 1}
	   |h^{i}_{\beta}(\xbar; 0)| \, \diff \xbar
    = \int_{\rho(\zbar) \geq 4^{-i}}
	   |h^0_{\gamma}(\xbar; 0)| \, \diff \xbar,
  \end{align*}
  when $\rho(\zbar) \leq 1$.

  Now define the function $\phi(\xbar) \defeq \sum_{i=1}^{\infty}
  \ind{\{\rho(\xbar) \geq 4^{i}\}}$ and observe that $\phi$ has
  sublinear growth.  We may then use Lemma~\ref{thm:4.6} to produce a
  function $\hhat_\mu : \bb R^d \rightarrow \bb R_+$ such that $N_4 =
  N_4(B,\mu) = \int_{\bb R^d} \phi(4,x) \hhat_\mu(x) \, \diff x <
  \infty$ and $h_{\alpha}(s,x; 0,0) \leq \hhat_\mu(x)$ for all $\alpha
  \in \scr A(\mu)$ and $|s| \in [1,4]$.
  We then see that
  \begin{align*} \label{eq:6}
    \sum_{i=-\infty}^{-1}
      \int_{\rho(\xbar) \geq N_3}
	|h^{i}_{\alpha}(\xbar; \zbar)| \, \diff \xbar
    &\leq \int_{\bb R} \int_{\bb R^d} \ind{[1,4]}(|s|) \, \phi(s,x) \hhat(x)
					\, \diff x \, \diff s
    \leq 6 N_4,
  \end{align*}
  when $\rho(\zbar) \leq 1$.  In particular, we have
  \begin{equation*}
    \sum_{i=-\infty}^{-1} \int_{\rho(\xbar) \geq N_3}
      |h^{i}_{\alpha}(\xbar, \ybar) - h^{i}_{\alpha}(\xbar, 0)| \, \diff \xbar
      \leq 12 \, N_4,
      \qquad \text{ when $\rho(\ybar) \leq 1$.}
  \end{equation*}
  We have now shown that~\eqref{eq:4.13} holds, so the proof is complete.
\end{proof}

The remaining case then follows easily by duality.
\begin{corollary} \label{thm:4.12} %
  Let $p \in (2,\infty)$.  Then there exists a constant $N= N(B,\mu,p)$
  such that $\|K^j_\alpha f\|_{L^p} \leq N \|f\|_{L^p}$ for all
  $\alpha \in \scr A(\mu)$, $j \in \bb N$, and $f \in C^\infty_K(\bb
  R^{1+d})$.
\end{corollary}
\begin{proof}
  Let $\alpha = (c,k,l,m) \in \scr A(\mu)$, set $\alpha^* =
  (c,k,l,1-m)$ and $q = p/(p-1) \in (1,2)$, and choose $N(B,\mu,q)$ as
  in Lemma~\ref{thm:4.11} such that $\|K^j_{\alpha^*} g\|_{L^q} \leq N
  \|g\|_{L^q}$ for all $g \in C^\infty_K(\bb R^{1+d})$.	 We then have
  \begin{align*}
    \int_{\bb R^{1+d}}
	  K^j_\alpha f(\xbar) g(\xbar) \, \diff \xbar
    &= \int_{\bb R^{1+d}}
	  f(\xbar) K^j_{\alpha^*} g(\xbar) \, \diff \xbar
    \\
    &\leq \|f\|_{L^p} \| K^j_{\alpha^*} g \|_{L^q}
    \leq N \|f\|_{L^p} \| g \|_{L^q}.
  \end{align*}
  Taking the supremum over $g$ with $\|g\|_{L^q} \leq 1$, we see that
  $\|K^j_\alpha f\|_{L^p} \leq N \|f\|_{L^p}$.
\end{proof}

The proof of Theorem~\ref{thm:4.1} now follows in a few lines.

\noindent
\textbf{Proof of Theorem~\ref{thm:4.1}} Choose any $p \in (1,
\infty)$, $1 \leq i,j \leq d_0$, and measurable $c : \bb R \rightarrow
\Smu$, and set $\alpha = (c,i,j,0)$.  Using Lemma~\ref{thm:4.10},
Lemma~\ref{thm:4.11}, or Corollary~\ref{thm:4.12}, we may find a constant
$N = N(B,\mu,p)$ such that $\|K^\ell_\alpha f\|_{L^p} \leq
\|f\|_{L^p}$ for all $\ell \in \bb N$ and $f \in C^\infty_K(\bb
R^{1+d})$.  Theorem~\ref{thm:3.6} asserts that the functions
$K^{\ell}_\alpha f$ converges to $\partial_{ij} G_{c,B} f$ pointwise
on $\bb R^{1+d}$ as $\ell \rightarrow \infty$, so we may invoke Fatou's
Lemma to conclude that $\|\partial_{ij} G_{c,B} f\|_{L^p} \leq \|f\|_{L^p}$
for all $f \in C^\infty_K(\bb R^{1+d})$.

\section{Uniqueness for the Martingale Problem} \label{sec:5}

We now use the estimate obtained in the previous section to obtain
uniqueness results for a class of degenerate martingale problems.
Given a law $\bb P$ on $C(\bb R_+; \bb R^d)$, we will refer to the
functionals $f \mapsto \bb E^{\bb P}[\int_0^T f(t,X_t) \, \diff t]$
informally as Green's functionals.  We will obtain an \textit{a
  priori} estimate for the Green's functionals associated with the
solutions to martingale problems in a particular class.  More
specifically, we will show that they are bounded functionals on
$L^p([0,T]\cross\bb R^d)$ for each $T$.  Once this is done, we may use
the estimates obtained in the previous section to obtain a local
uniqueness result.  We will then extend this uniqueness result using a
localization procedure. Finally, we will relax the drift conditions by
employing a second localization step.

In the previous section, it was convenient to work with the
operator $G_{c,B}$ which operated on functions in $C_K(\bb
R^{1+d})$.  We would now prefer to work with a restricted version
of $G_{c,B}$ which operates on functions in $C_K([0,T)\cross\bb
R^{d})$.  Given a function $c : \bb R_+ \rightarrow \bb R^d$, we define
\begin{equation*}
  G^T_{c,B} f(s,x) = \int_0^T \int_{\bb R^d} p_{c,B}(s,x;t,y) f(t,y)
			  \, \diff y \, \diff t,
  \qquad (s,x) \in [0,T] \cross \bb R^{d}.
\end{equation*}

We start by giving some estimates for the operator $G^T_{c,B}$.	 Recall
that $\dbar$ denotes the homogeneous dimension of the group associated
with the matrix $B$.

\begin{lemma} \label{thm:5.1} %
  Let $c : \bb R_+ \rightarrow \Smu$, and let $i,j \in
  \{1,\dots,d_0\}$, $p \in (1,\infty)$, and $f \in
  C^\infty_K([0,T)\cross \bb R^d)$.  Then $\|\partial_{ij}G^T_{c,B}
  f\|_{L^p([0,T]\cross\bb R^d)} \leq N(B,\mu,p)
  \|f\|_{L^p([0,T]\cross\bb R^d)}$.
\end{lemma}
\begin{proof}
  Let $\tilde f \in C^\infty_K(\bb R^{1+d})$ with $\tilde f(s,x) =
  f(s,x)$ when $s \in [0,T)$ and $\tilde f(s,x) = 0$ when $s \geq T$.
  The existence of such an extension is shown in \cite{seeley1964ecfd}.
  Now let $\phi \in C^\infty(\bb R;[0,1])$ with $\ind{[0,\infty)} \leq
  \phi \leq \ind{[-1,\infty)}$, let $N = N(B,\mu,p)$ denote the
  constant obtained in Theorem~\ref{thm:4.1}, and set $\tilde
  f_n(s,x) = \phi(n s) \tilde f(s,x)$.	Then $G^T_{c,B} f(s,x) = G_{c,B}
  \tilde f_n(s,x)$ for all $s \in [0,T]$ and $n \in \bb N$, and
 \begin{align*}
   \|\partial_{ij} G^T_{c,B} f\|_{L^p([0,T)\cross\bb R^d)}
   &\leq
   \liminf_{n \rightarrow \infty}
     \|\partial_{ij}  G_{c,B} \tilde f_n\|_{L^p(\bb R^{1+d})}
   \\
   &\leq \liminf_{n \rightarrow \infty} N \|\tilde f_n\|_{L^p(\bb R^{1+d)}}
   = N \| f\|_{L^p([0,T)\cross\bb R^d)},
  \end{align*}
  so the result follows.
\end{proof}
\begin{lemma} \label{thm:5.2} %
  Let $f \in C_K([0,T)\cross \bb R^d)$, let $c : \bb
  R_+ \rightarrow \Smu$ be measurable, and let $p \in (\dbar/2, \infty)$.
  Then
  \begin{equation} \label{eq:5.1}
    \bigl\|G_{c,B}^T f\bigr\|_{L^\infty([0,T)\cross \bb R^d)}
    \leq N(B, \mu, p) \, T^{1-\dbar/2p}
    \, \|f\|_{L^p([0,T)\cross \bb R^d)}.
  \end{equation}
\end{lemma}
\begin{proof}
  Let $\underline c(t) = I^{d_0} / \mu$ and $\overline c(t) = \mu
  I^{d_0}$, so $C_{\underline c}(s,t) \leq C_{c}(s,t) \leq
  C_{\overline c}(s,t)$ for all $t > s \geq 0$.	 The functions
  $\underline c$ and $\overline c$ are translation and dilation
  invariant, so we have (see \eqref{eq:3.5})
  \begin{equation*}
    \frac{\det C_{\underline c}(s,t)}{\det C_{\overline c}(s,t)}
    = \frac{\det \{\dil_{(t-s)}^{1/2} \, C_{\underline c}(0,1)
		   \, \dil_{(t-s)}^{1/2} \}}
	   {\det \{\dil_{(t-s)}^{1/2} \, C_{\overline c}(0,1)
		   \, \dil_{(t-s)}^{1/2}\}}
    = \frac{\det C_{\underline c}(0,1)}{\det C_{\overline c}(0,1)}.
  \end{equation*}
  In particular, if we set $N_1 = \det C_{\underline
    c}(0,1) / \det C_{\overline c}(0,1) > 0$, then we have
  $p_{c,B}(s,x;t,y) \leq N_1^{-1/2} \, p_{\overline c}(s,x;t,y)$.

  Now, if $q \in (1, 1+2/(\dbar-2))$ and $p = q / (q-1) > \dbar/2$,
  then
  \begin{align*}
    \int_{\bb R^d} \{p_{c,B}(s,x; t,y)\}^q \, \diff y
    &\leq N_1^{-q/2} \, \int_{\bb R^d} \{p_{\overline c}(0, 0; t-s, z)\}^q
				       \, \diff z
    \\
    &= N_1^{-q/2} (t-s)^{-(q-1)(\dbar-2)/2}
       \int_{\bb R^d} \{p_{\overline c}(0,0; 1, z)\}^q \diff z,
  \end{align*}
  where the first inequality follows by left-translation (see
  \eqref{eq:3.4}), the second equality follows by dilation
  (see~\eqref{eq:3.6}), and the last integral is finite.  We now
  observe that $q (1 - \dbar/2p) = 1 - (q-1)(\dbar-2)/2$, so
  \begin{align*}
    \int_s^T \int_{\bb R^d} \{p_{c,B}(s,x; u,z)\}^q \, \diff z \, \diff u
    \leq \frac{N_1^{-q/2}}{q (1 - \dbar/2p)} T^{q (1 - \dbar/2p)}
	 \int_{\bb R^d} \{p_{\overline c}(0,0; 1, z\}^q \diff z,
  \end{align*}
  and~\eqref{eq:5.1} follows by duality.
\end{proof}

The next step is to show that the Green's functionals can be expressed
in terms of $G_{c,B}^T$ and a stochastic correction term.

\begin{lemma} \label{thm:5.3} Let $a : \bb R_+ \cross C(\bb R; \bb
  R^{d}) \rightarrow S^{d_0}_+$ be $\bb C^d$-progressive, let $c : \bb
  R \rightarrow \Smu$ be measurable, let $f \in C^\infty_K([0,T)\cross
  \bb R^d)$.  Also let $\bb P$ denote a solution to the $(a(X), B
  X)$-martingale problem starting at $(s,x) \in [0,T) \cross \bb R^d$,
  and define the process $\phi_t(a,c,B,f) = \frac{1}{2}
  \sum_{i,j=i}^{d_0} \bigl\{a_t^{ij} - c^{ij}(t)\bigr\}
  \partial_{ij} G_{c,B}^T f (t, X_t)$ for $t \in [s,T]$.
  Then
  \begin{align}
    \label{eq:5.2}
    \bb E^{\bb P}\Bigl[\int_s^T f(t,X_t) \, \diff t\Bigr]
    = G_{c,B}^T f(s,x)
      + \bb E^{\bb P}\Bigl[\int_s^T \phi_t(a,c,B,f) \, \diff t\Bigr].
  \end{align}
\end{lemma}
\def \linftyT{{L^\infty([0,T]\cross \bb R^d)}}
\begin{proof}
  As $G_{c,B}^Tf(t,x) = 0$ for $t = T$, it is enough to
  show that
  \begin{equation} \label{eq:5.3}
    M_t = G_{c,B}^T f(t, X_t)
	     + \int_s^t f(u, X_u) - \phi_u(a,c,B,f) \, \diff u,
    \quad t \in [s,T].
  \end{equation}
  is a martingale.  Using Lemma~\ref{thm:3.6}, we may find a constant
  $N = N(B, f)$ such that $G_{c,B}^Tf \leq N$ and $\partial_{ij}
  G_{c,B}^Tf \leq N$ on $[0,T]\cross\bb R^d$ when $1 \leq i,j \leq
  d_0$.	 Although Lemma~\ref{thm:3.6} is stated in terms of $G_{c,B}$,
  one may repeat the argument given in Lemma~\ref{thm:5.1} to see that
  the conclusions also holds for $G^T_{c,B}$.  As a result, $M$ is
  bounded and the bound does not depend upon $c$.  If $c \in
  C^\infty(\bb R_+; \Smu)$, then we also have $\partial_s G_{c,B}^T f
  + L^{c,B} G_{c,B}^T f= -f$ by the same lemma, and the result follows
  from Ito's Lemma.  To handle the general case, we choose a sequence
  $c_n \in C^\infty(\bb R_+; \Smu)$ with $c_n \rightarrow c$ in
  $L^p([0,T]\cross\bb R^d)$, and we let $M^n$ denote the process
  obtained by replacing $c$ with $c_n$ in \eqref{eq:5.3}.  Then $M^n$
  is a uniformly bounded sequence of martingales that converges
  pointwise to $M$, so we may conclude that $M$ is a martingale.
\end{proof}

We now produce the desired estimate for the Green's functionals.  We
do this by imposing conditions which ensure that the stochastic
correction term in the previous lemma is sufficiently small.
\begin{lemma} \label{thm:5.4} %
  Let $a : \bb R_+ \cross C(\bb R_+; \bb R^d) \rightarrow \Smu$ be
  $\bb C^d$-progressive, let $c : \bb R_+ \rightarrow \Smu$ be
  measurable, let $p \in (\dbar/2, \infty)$, and suppose that $\bb P$
  is a solution to the $(a(X), B X)$-martingale problem starting
  at $(s,x)$.  Then there exists constants $N = N(d,\mu,B,p)$ and
  $\varepsilon = \varepsilon(d,\mu,B,p) > 0$ such that
  \begin{equation} \label{eq:5.4}
    \bb E^{\bb P}\Bigl[\int_s^T |f(t, X_t)| \, \diff t\Bigr]
    \leq N T^{1-\dbar/2p} \|f \|_{L^p([0,T)\cross \bb R^d)},
  \end{equation}
  for all $T \geq 0$ and $f \in C^\infty_K([0,T)\cross\bb R^d)$ when
  \begin{equation} \label{eq:5.5}
    \|a(s,x) - c(s)\| \leq \varepsilon,
    \qquad \text{for all $(s,x) \in \bb R_+ \cross \bb R^d$.}
  \end{equation}
\end{lemma}
\begin{proof}
  The proof of this lemma is somewhat involved and follows in the same
  way as Lemmas~7.1.2, 7.1.3, and 7.1.4 of \cite{stroock1979mdp}, so
  we only recall the main ideas for the convenience of the reader.
  First we consider the case where $a$ is a simple process with
  respect to a deterministic time partition and~\eqref{eq:5.5} may not
  hold.	 In this case, one can show that~\eqref{eq:5.4} holds when $N$
  is replaced by a constant $N_1$ which depends upon the number
  of points in the time partition.  This is done by conditioning on
  the information available at the start of each time interval in the
  partition and applying the estimate~\eqref{eq:5.1} to the
  conditioned process.

  We continue to consider the case where $a$ is a simple process, but
  we now produce a constant which does not depend upon the number of
  steps in the partition.  Set $\varepsilon = \sup_{(s,x) \in \bb R_+
    \cross \bb R^{d}} |a(s,x) - c(s)|$ and let $N_2 = N_2(\bb P)$
  denote the smallest constant such that~\eqref{eq:5.4} holds when $N$
  is replaced by $N_2$.	 The previous step ensures that $N_2$ is
  finite.  We will now show that $N_2$ is bounded by a constant which
  does not depend upon $\bb P$ when $\varepsilon$ is sufficiently
  small.  First we use Lemma~\ref{thm:5.2} and Lemma~\ref{thm:5.3} to
  produce a constant $N_3 = N_3(B, \mu, p)$ such that
  \begin{align*}
    \bb E^{\bb P}\Bigl[\int_s^T |f(t, X_t)| \, \diff t\Bigr]
    &\leq N_3 T^{1-\dbar/2p} \|f\|_{L^p([0,T)\cross\bb R^d)} \\
    & \qquad \quad
	 + \half \varepsilon N_2 T^{1-\dbar/2p}
	   \sum_{i,j=1}^{d_0}
	   \|\partial_{ij} G_{c,B}^T f\|_{L^p([0,T)\cross\bb R^d)}.
  \end{align*}
  We next apply Theorem~\ref{thm:4.1} to produce a constant $N_4 =
  N_4(d, \mu, B, p)$, such that
  \begin{align*}
      \bb E^{\bb P}\Bigl[\int_s^T |f(t, X_t)| \, \diff t\Bigr]
      &\leq (N_3 + \half \varepsilon d_0^2 N_2 N_4) T^{1-\dbar/2p}
	    \|f\|_{L^p([0,T)\cross\bb R^d)}
  \end{align*}
  We then take the supremum over the set of functions $f \in
  C^\infty_K([0,T)\cross \bb R^d)$ with $\|f\|_{L^p([0,T)\cross\bb
    R^d)} \leq 1$ to see that $N_2 \leq 2 N_3$ when $\varepsilon \leq
  1 / (d_0^2 N_4)$.  This gives a bound which depends upon $B$, $\mu$,
  and $p$ but does not depend upon the number of time steps in the
  partition. The general case may then be handled by approximation.
\end{proof}

\begin{remark} \label{thm:5.5} %
  When $d = d_0$, $a$ is of rank $d$ everywhere and much stronger
  results are available.  Let $a : \bb R_+ \cross \bb R^d \rightarrow
  S^d_{\mu}$ and $b : \bb R_+ \cross \bb R^d \rightarrow \bb R^d$ be
  bounded, measurable functions.  Then Krylov \cite{krylov1971ceft}
  has shown that there exists a constant $N$ which depends upon $d$,
  $\mu$, $\|b\|_{L^\infty(\bb R_+\cross \bb R^d)}$, and $T$ such that
  \begin{equation} \label{eq:5.6}
    \bb E^{\bb P}\Bigl[\int_0^T |f(t, X_t)| \, \diff t\Bigr]
    \leq N \|f \|_{L^{1+d}([0,T]\cross \bb R^d)},
  \end{equation}
  for any solution $\bb P$ to the $(a,b)$-martingale problem.  In
  particular, this result does not require $a$ to be well-approximated
  by a deterministic function of time.	One consequence of this
  estimate is the existence of weak solutions to SDEs with measurable,
  uniformly positive-definite covariance.  The reader may consult
  Theorem 2.3.4 and Theorem 2.6.1 of \cite{krylov1980cdp} for the
  proof of these results.  Another consequence of the
  estimate~\eqref{eq:5.6} is that weak uniqueness holds for SDEs in
  which the covariance function is VMO-continuous in the spacial
  variables.  The reader may consult Remark~2.2 of
  \cite{krylov2007peev} for a brief discussion of this fact.

  The estimate~\eqref{eq:5.6} is closely related the
  Aleksandrov-Bakelman-Pucci estimate of PDE theory and the parabolic
  extension due to Krylov and Tso.  These results depend in an
  essential way upon the geometry of convex functions in $\bb R^d$,
  and analogous results are not yet available for our geometric
  setting. One can consult Section 9 of \cite{garofalo2006npcf} for a
  further discussion of these issues in the context of the Heisenberg
  group.  The lack of such an estimate is the main impediment to
  obtaining a weak uniqueness result using the estimates obtained by
  Bramanti, Cerutti, and Manfredini in \cite{bramanti1996les}.
\end{remark}

We are dealing with martingale problems where the drift is unbounded,
but linear, so the next two lemmas will prove useful.

\begin{lemma} \label{thm:5.6} %
  Let $a^\alpha : \bb R_+ \cross \Omega^\alpha \rightarrow S^d_+$ and
  $b^\alpha : \bb R_+ \cross \Omega^\alpha \rightarrow \bb R^d$ be
  progressive processes, possibly defined on different spaces, and
  suppose that $X^\alpha$ is a continuous solution to the $(a^\alpha,
  b^\alpha)$-martingale problem for each $\alpha \in A$.  Further
  suppose that the random variables $\{X^\alpha_0\}_{\alpha\in A}$ are
  uniformly bounded, and that there exists a constant $N$ such that
  \begin{equation} \label{eq:5.7}
    \|a^\alpha_t\| + |b^\alpha_t|^2
    \leq N \bigl\{1 + \sup_{s \leq t} \, (X^\alpha_s)^2\bigr\},
    \qquad t \in\bb R_+, \, \alpha \in A.
  \end{equation}
  Then the collection of processes $\{X^\alpha\}_{\alpha \in A}$ is
  tight.
\end{lemma}
\begin{proof}
  Let ${T^\alpha_n}$ denote the stopping time ${T^\alpha_n} = \inf\{ t
  \in \bb R_+ : |X^\alpha_t| \geq n \}$, and set $X^{\alpha,n} =
  (X^\alpha)^{T^\alpha_n}$. The process $X^{\alpha,n}$ is a
  solution to the $(\ind{[0, T^\alpha_n]}a^\alpha, \ind{[0,
    T^\alpha_n]}b^\alpha)$-martingale problem, and the processes
  $\ind{[0, T^\alpha_n]}a^\alpha$ and $\ind{[0, T^\alpha_n]}b^\alpha$
  are uniformly bounded with respect to $\alpha$, so we may conclude
  that the collection of processes $\{X^{\alpha,n}\}_{\alpha \in A}$
  is tight for each fixed $n$.

  It then follows from \eqref{eq:5.7}, the uniform boundedness of the
  random variables $\{X^\alpha_0\}_{\alpha\in A}$, the
  Burkholder-Davis-Gundy inequalities, Gronwall's Lemma, and
  Chebyshev's inequality that
  \begin{equation} \label{eq:5.8}
    \lim_{n \rightarrow \infty} \;
    \sup_{\alpha \in A} \; \bb P^\alpha[T^\alpha_n \leq t ] = 0,
    \qquad \text{for each $t \geq 0$.}
  \end{equation}
  The tightness of the collection $\{X^\alpha\}_{\alpha \in A}$ then
  follows from~\eqref{eq:5.8} and the tightness of
  $\{X^{\alpha,n}\}_{\alpha \in A}$ for each fixed $n$.
\end{proof}

\begin{lemma} \label{thm:5.7} %
  Let $(s,x) \in \bb R_+ \cross \bb R^d$, let $a : \bb R_+ \cross
  C(\bb R_+ \cross \bb R^d) \rightarrow S^d_+$ and $b : \bb R_+ \cross
  C(\bb R_+ \cross \bb R^d) \rightarrow S^d_+$ be $\bb
  C^d$-progressive, and suppose that the maps $\omega \rightarrow a(t,
  \omega)$ and $\omega \rightarrow b(t, \omega)$ are continuous for
  each fixed $t \geq 0$.  Further suppose that there exists a constant
  $N$ such that
  \begin{equation*}
    \|a_t\| + |b_t|^2 \leq N \bigl(1 + (X^*_t)^2\bigr),
    \qquad t \in \bb R_+.
  \end{equation*}
  Then there exists a solution to the $(a, b)$-martingale
  problem starting at $(s,x)$.
\end{lemma}
\begin{proof}
  When $a$ and $b$ are bounded, the result follows by approximating
  the processes $a$ and $b$ using an Euler-type scheme and checking that
  any weak limit point of the approximations is a solution to the
  desired martingale problem.  The reader may consult Theorem~6.1.6 of
  \cite{stroock1979mdp} for the details.  The general case follows by
  truncating the coefficients and using the previous tightness result
  to find a limit point that solves the desired martingale problem.
\end{proof}

We now use this existence result to show that the localized martingale
problem is well-posed.

\begin{lemma} \label{thm:5.8} %
  Let $a : \bb R_+\cross \bb R^d \rightarrow \Smu$, $b: \bb R_+ \cross
  \bb R^{d} \rightarrow \bb R^{d}$, and $c : \bb R_+ \rightarrow \Smu$
  be bounded, measurable functions, and suppose that $b^i(t,x) = 0$
  for $i > d_0$.  Then there exists a constant $\varepsilon =
  \varepsilon(B, \mu) > 0$ such that the $(a(X), b(X) +
  BX)$-martingale problem is well-posed when
  \begin{equation} \label{eq:5.9}
    \|a(s,x) - c(s)\| \leq \varepsilon,
    \qquad \text{for all $(s,x) \in \bb R_+ \cross \bb R^d$.}
  \end{equation}
\end{lemma}
\begin{proof}
  As $a$ is uniformly positive definite and $b$ is bounded, we may use
  Girsanov's Theorem to place solutions to the $(a(X), b(X) +
  BX)$-martingale problem in one-to-one correspondence with solutions
  to the $(a(X), BX)$-martingale problem.  As a result, we may assume
  without loss of generality that $b = 0$.  Now fix some $p >
  \dbar/2$, and let $N_1 = N(B, \mu, p)$ and $\varepsilon_1 =
  \varepsilon_1(B, \mu, p)$ denote the constants obtained in
  Lemma~\ref{thm:5.4}.

  We now show the existence of a solution for each initial
  condition when $\varepsilon \leq \varepsilon_1$.
  By mollification in the spacial directions, we may find a sequence
  of functions $a_n : \bb R_+\cross\bb R^d \rightarrow \Smu$ such that
  $\lim_{n\rightarrow \infty} \|a - a_n\|_{L^p([0,T] \cross \bb R^d)}
  = 0$ for each fixed $T$; $\sup_{(s,x) \in \bb R_+\cross\bb R^d}
  \|a_n(s,x) - c(s)\| \leq \varepsilon$; and the functions $x \mapsto
  a(s,x)$ are continuous for each fixed $s \geq 0$.  The existence of a
  solution $\bb P_n$ to the $(a_n(X), BX)$-martingale problem starting
  at $(s,x) \in \bb R_+\cross\bb R^d$ then follows from
  Lemma~\ref{thm:5.7}.	The collection $\{\bb P_n\}_{n < \infty}$ is
  tight by Lemma~\ref{thm:5.6}, so may find a weak limit point $\bb
  P_\infty$.  The inequality~\eqref{eq:5.4} holds with $\bb P$ replaced
  by $\bb P_n$ for any $n$, so it also holds with $\bb P$ replaced by
  $\bb P_\infty$.  As a result, we have
  \begin{equation*}
      \bb E^{\bb P_m}\Bigl[\int_s^T \|a_n(u, X_u) - a_\ell(u, X_u)\|
				    \, \diff u\Bigr]
      \leq N T^{1 - \dbar/2p} \|a_n - a_\ell\|_{L^p([0,T]\cross \bb R^d)},
  \end{equation*}
  for any $\ell,m,n \in \bb N\cup\{\infty\}$ if we set
  $a_\infty = a$.  It then follows easily that $\bb P_\infty$ is a
  solution to the $(a(X), BX)$-martingale problem.

  We now show that the solution to the $(a(X),BX)$-martingale problem
  with initial condition $(s,x) \in \bb R_+ \cross \bb R^d$ is unique.
  Choose and $T > s$ and let $K^T$ denote the operator
  \begin{equation*}
    K^Tf(s,x)
    = \frac{1}{2} \sum_{i,j=1}^d \{a^{ij}(s,x) - c^{ij}(s)\}
				 \, \partial_{ij} G^T_{c,B} f,
  \end{equation*}
  for $f \in C^\infty_K([0,T)\cross\bb R^d)$.  Now let $U^T :
  L^p([0,T]\cross\bb R^d) \rightarrow L^p([0,T]\cross\bb R^d)$ denote
  the unique, bounded operator which agrees with $K^T$ on
  $C^\infty_K([0,T)\cross\bb R^d)$, and let $V^T_{s,x} :
  L^p([0,T]\cross\bb R^d) \rightarrow \bb R$ denote the unique,
  bounded functional which agrees with the map $f \mapsto G_{c,B}^T
  f(s,x)$ on $C^\infty_K([0,T)\cross\bb R^d)$.	The existence and
  uniqueness of these extensions follows from Lemma~\ref{thm:5.1} and
  Lemma~\ref{thm:5.2}.	Moreover, it follows from Lemma~\ref{thm:5.2},
  that we may choose $\varepsilon_2 = \varepsilon_2(B,\mu,p) > 0$ so
  small that the operator $I - U^T$ is invertible.

  Set $\varepsilon = \varepsilon_1 \minn \varepsilon_2$ and
  assume~\eqref{eq:5.9} holds.	Now fix any $f_\infty \in
  C^\infty_K([0,T)\cross\bb R^d)$ and choose $f_n \in
  C^\infty_K([0,T)\cross\bb R^d)$ such that $f_n$ converges to $(I-
  U^T)^{-1}f_\infty$ in $L^p([0,T]\cross\bb R^d)$.  It
  follows from Lemma~\ref{thm:5.3} that
  \begin{equation*}
    \bb E^{\bb P}\Bigl[\int_s^T (I - K^T)f_n(t, X_t) \, \diff t\Bigr]
    = G_{c,B}^T \, f_n(s,x),
  \end{equation*}
  for each $n$.	 We may then use Lemma~\ref{thm:5.4} to pass to the
  limit in this expression and obtain
  \begin{equation} \label{eq:5.10}
    \bb E^{\bb P}\Bigl[\int_s^T f_\infty(t, X_t) \, \diff t\Bigr]
    = V^T_{s,x} (I- U^T)^{-1} f_\infty.
  \end{equation}
  As \eqref{eq:5.10} holds for any solution to the
  $(a(X),BX)$-martingale problem with initial condition $(s,x)$ and the
  right-hand side does not depend upon $\bb P$, we may conclude that
  the solution to the $(a(X),BX)$-martingale problem starting at
  $(s,x)$ is unique (e.g.\ Cor.\ 6.2.5 of \cite{stroock1979mdp}).
\end{proof}

To extend the local result to a global result, we use a localization
procedure due to Stroock and Varadhan.	The main technical difficulty
that we encounter here is that the drift in our local solutions must
be unbounded; we cannot truncate the map $x \mapsto Bx$ without
disturbing the geometrical structure.  As a result, we need to extend
the localization results provided by Stroock and Varadhan to allow for
unbounded coefficients.	 We should also point out that when we make
use this result later in Theorem~\ref{thm:5.14}, we will not have any
uniform control on the growth of the local coefficients $(a_\alpha,
b_\alpha)_{\alpha \in \scr A}$, so it is important that the following
result only imposes conditions on the functions $a$ and $b$.

\begin{theorem} \label{thm:5.9} %
  Let $a : \bb R_+\cross \bb R^d \rightarrow S^{d}_+$ and $b : \bb
  R_+\cross \bb R^d \rightarrow \bb R^{d}$ be locally bounded,
  measurable functions, and suppose that there exists a constant $N$
  such that
  \begin{equation*}
    \|a(s,x)\| + |b(s,x)|^2 \leq N (1 + |x|^2),
    \qquad \text{for all $(s,x) \in \bb R_+\cross\bb R^d$.}
  \end{equation*}
  Let $\{G_\alpha\}_{\alpha\in A}$ be an open cover of $\bb R_+\cross
  \bb R^d$, and suppose that each open set in $\{G_\alpha\}_{\alpha\in
    A}$ is associated with a pair of locally bounded, measurable
  functions $a_\alpha : \bb R_+\cross \bb R^d \rightarrow S^{d}_+$ and
  $b_\alpha : \bb R_+\cross \bb R^d \rightarrow \bb R^{d}$ such that
  $a_\alpha = a$ and $b_\alpha = b$ on the set $G_\alpha$, and the
  $(a_\alpha, b_\alpha)$-martingale problem is well-posed.  Then the
  $(a,b)$-martingale problem is well-posed.
\end{theorem}
\begin{proof}
  As $\bb R_+\cross \bb R^d$ is $\sigma$-compact, we may always find a
  countable subcover, so we may assume that $A = \bb N$.  For
  notational convenience, we will show the existence and uniqueness of
  a solution to $(a,b)$-martingale problem starting at $(0,0) \in \bb
  R_+ \cross \bb R^d$; however, it will be clear that the same
  argument works for any initial condition.

  For each compact set $K \subset \bb R_+\cross \bb R^d$, let $n(K)$
  denote the smallest natural number such that $K \subset
  \cup_{i=1}^{n(K)} G_i$, and define
  \begin{equation*}
    \rho_0(K) = \inf_{\xbar \in K} \; \max_{i \leq n(K)}
				   \; \dist(\xbar, G^c_i),
  \end{equation*}
  where $G^c_i$ denotes the complement of the set $G_i$.
  Observe that $\rho_0(K) > 0$ for each compact set $K$.  Now define
  \begin{equation} \label{eq:5.11}
    \rho(t,r) = 1 \minn
                \sup_{u \geq t, q\geq r} \;
                \rho_0\bigl([0,u]\cross \overline  B^d_q(0)\bigr) / 2,
              \qquad t,r \geq 0.
  \end{equation}
  We introduce the supremum to ensure that the function $\rho$ is
  nonincreasing in each coordinate.  Finally, let $\phi(s,x)$ denote
  the smallest natural number such that $\dist((s,x), G^c_{\phi(s,x)})
  > \rho(s,|x|)$, so $B^{1+d}_{\rho(s,x)}(s,x) \subset G_{\phi(s,x)}$
  for all $(s,x) \in \bb R_+ \cross \bb R^d$.  It not hard to check
  that $\phi$ is finitely-valued and measurable.

  We will proceed by patching together local solutions, so we will
  need notion for the concatenation of measures.  The notation that we
  adopt here is a slight modification of the notation adopted in
  Chapter 6 of \cite{stroock1979mdp}.  If $\bb P$ is a probability
  measure on $C(\bb R_+; \bb R^d)$, $T$ is a $\bb C^d$-stopping time,
  $Y$ is an $\bb N$-valued, $\scr C_T^d$-measurable random variable, and
  $\bb Q : \bb N \cross \bb R_+ \cross \bb R^d \cross \scr C^d
  \rightarrow [0,1]$ is a measurable probability kernel with
  \begin{equation} \label{eq:5.12}
    \bb Q^n_{s,x}[X_t = x, \; t \leq s] = 1,
    \qquad n \in \bb N,\, (s,x) \in \bb R_+ \cross \bb R^d,
  \end{equation}
  then we let $\bb P \otimes_{T,Y} \bb Q$ denote the probability measure
  on $C(\bb R_+; \bb R^d)$ which is uniquely characterized by the
  following properties:
  \begin{enumerate}
  \item If $f : C(\bb R_+; \bb R^d) \rightarrow \bb R$ is a bounded,
    measurable function, then we have $\bb E^{\bb P \otimes_{T,Y} \bb
      Q}[f(X^T)] = \bb E^{\bb P}[f(X^T)]$.
  \item If $f : C(\bb R_+; \bb R^d) \rightarrow \bb R$ is a bounded,
    measurable function, then the random variable
    \begin{equation} \label{eq:5.13}
      \omega \rightarrow
      \bb E^{Q^{Y(\omega)}_{T(\omega), X_T(\omega)}}\Bigl[
	   f\bigl(X(\omega) \ind{[0, T(\omega))}
	   + X \, \ind{[T(\omega),\infty)}\bigr)\Bigr]
    \end{equation}
    is a version of $\bb E^{\bb P \otimes_{T,Y} Q}[f(X) \cond \scr C_T]$.
  \end{enumerate}
  Intuitively, $\bb P \otimes_{T,Y} \bb Q$ corresponds to the law of a
  process that begins evolving according to $\bb P$ prior to time $T$.
  At time $T$, a new law is selected from the collection $\{\bb
  Q^n_{s,x}\}$ using the random variables $T$, $X_T$, and $Y$, and the
  process then begins evolving in accordance with this new law.
  Observe that, as a consequence of~\eqref{eq:5.12}, the process
  $X(\omega) \ind{[0, T(\omega))} + X \, \ind{[T(\omega),\infty)}$ is
  $\bb Q^{Y(\omega)}_{T(\omega), X_T(\omega)}$-a.s. continuous, and
  the right-hand side of~\eqref{eq:5.13} is well-defined.

  We now begin patching together measures to produce a solution to the
  $(a(X), b(X))$-martingale problem starting at $(0,0)$.  To do this,
  we inductively define the sequence of stopping times: $T_0 = 0$ and
  \begin{equation*}
    T_{n} = \inf\bigl\{t \geq T_{n-1} : t - T_{n-1} +
		|X_t - X_{T_{n-1}}| \geq \rho(T_{n-1}, |X_{T_{n-1}}|)\bigr\},
  \end{equation*}
  for $n \geq 1$.  This is the hitting time of a closed set, so it is
  a $\bb C^d$-stopping time, even though $\bb C^d$ is not
  right continuous.  We then define the $\scr C_{T_n}$-measurable
  random variables $Y_0 =\phi(0,0)$ and $Y_n = \phi(T_n, X_{T_n})$ for
  $n \geq 1$.  For each $n$, we may find a measurable kernel $\bb Q^n$
  such that $\bb Q^n_{s,x}$ is the unique solution to the $(a_n(X),
  b_n(X))$-martingale problem starting at $(s,x)$. The measurability
  of the map $(s,x) \mapsto \bb Q^n_{s,x}$ follows immediately from
  the fact that the $(a_n(X), b_n(X))$-martingale problem is
  well-posed.  The reader may consult Exercise~6.7.4 of
  \cite{stroock1979mdp} for a proof of this fact.  We now inductively
  define a sequence of probability measures: $\bb P^0 = \bb
  Q^{\phi(0,0)}_{0,0}$, and $\bb P^n = \bb P^{n-1} \otimes_{T_n, Y_n}
  \bb Q$.

  We may characterize the measures $\bb P^n$ as solutions to
  martingale problems.	Define the processes
  \begin{align*}
    A^n_t &= \ind{[0,T_{n+1}]} a(t,X_t)
	     + \ind{(T_{n+1}, \infty)} a_{Y^{n}}(t,X),
    \\
    B^n_t &= \ind{[0,T_{n+1}]} b(t,X_t)
	     + \ind{(T_{n+1}, \infty)} b_{Y^{n}}(t,X).
  \end{align*}
  The pairs of functions $(a,b)$ and $(a_{Y^{n}(\omega)},
  b_{Y^{n}(\omega)})$ agree on the set $G_{Y^{n}(\omega)}$ and
  $X_t(\omega) \in G_{Y^{n}(\omega)}$ for all $t \in
  [T^{n}(\omega), T^{n+1}(\omega)]$.  It then follows as in the proof
  of Lemma~6.6.4 of \cite{stroock1979mdp} that $\bb P^n$ is a
  solution to the $(A^n, B^n)$-martingale problem starting at $(0,0)$.

  We now show that
  \begin{equation} \label{eq:5.14}
    \lim_{n\rightarrow\infty} \bb P^{n}[T_{n+1} \leq t] = 0,
    \qquad \text{for each $t > 0$.}
  \end{equation}
  Fix $t > 0$ and $\varepsilon > 0$ and set $X^n = X^{T_{n}}$.	It follows
  from the previous characterization of $\bb P^n$, that $X^{n+1}$ is a
  solution to the $(\ind{[0,T_{n+1}]} a(X^n), \ind{[0,T_{n+1}]}
  b(X^n))$-martingale problem starting at $(0,0)$ under $\bb P^n$.  It
  then follows from Lemma~\ref{thm:5.6} that the sequence $\scr L(X^{n+1}
  \mid \bb P^n)$ is tight.  In particular, we may choose $M$ and then
  $\delta > 0$ such that
  \begin{align}
    \label{eq:5.15}
    \bb E^{n}[ (X^{n+1})^*_t \geq M] \leq \varepsilon/2,
    \qquad \text{ for all $n \in \bb N$,}
    \\
    \label{eq:5.16}
    \bb E^{n}[ m^\delta_t(X^{n+1}) \geq \rho(t, M)/2] \leq \varepsilon/2,
    \qquad \text{ for all $n \in \bb N$,}
  \end{align}
  where $m^\delta_t$ denotes the modulus of continuity
  \begin{equation} \label{eq:5.17}
    m^\delta_t(\omega) = \sup_{0 \leq r \leq s \leq (r + \delta) \minn t}
		    |\omega(s) - \omega(r)|,
    \qquad \omega \in C(\bb R_+; \bb R^d).
  \end{equation}

  We now show that
  \begin{equation} \label{eq:5.18} %
    \begin{aligned}[c]
      \lefteqn{ \{T_n \leq t \} \cap \{ X^*_{T_n} < M \}
                \cap \bigl\{ m^\delta_t(X^n) < \rho(t,M)/2 \bigr\} }
      \qquad \qquad \qquad \qquad \qquad \qquad \qquad \qquad \\
      &\subset \bigl\{ T_{n} \geq n \bigl(\delta \minn \rho(t,M)/2\bigr) \bigr\}.
    \end{aligned}
  \end{equation}
  If we suppose that $X^*_{T_n} < M$, then the process $X$ remains in
  the ball $B^d_M(0)$ during $[0, T_{n}]$.  If we further suppose that
  $T_n \leq t$, then we must have $T_m - T_{m-1} + |X_{T_m} -
  X_{T_{m-1}}| \geq \rho(t,M)$ for each $m \leq n$ because $\rho$ is
  nonincreasing.  In particular, we must have $T_m - T_{m-1} \geq
  \rho(t,M)/2$ or $|X_{T_m} - X_{T_{m-1}}| \geq \rho(t,M)/2$.  But if
  $m^\delta_t(X^n) < \rho(t,M)/2$, then it takes the process $X^n$ at
  least $\delta$ units of time to move by the amount $\rho(t,M)/2$, so
  we must have $T_m - T_{m-1} > \delta$ when $|X_{T_m} - X_{T_{m-1}}|
  \geq \rho(t,M)/2$.  Either way, we have $T_m -
  T_{m-1} > \delta \minn \rho(t,M)/2$ for all $m \leq n$ and
  summing over $1 \leq m \leq n$ gives~\eqref{eq:5.18}.

  It then follows immediately from~\eqref{eq:5.18}, that
  \begin{equation*} 
    \{ (X^n)^*_t < M \} \cap \{ m_t^\delta(X^n) < \rho(M) \}
    \subset
    \bigl\{ T_{n} \geq n \bigl(\delta \minn \rho(t,M)/2\bigr) \bigr\}
    \cup \{ T_{n} > t \}.
  \end{equation*}
  If we choose $n$ so large that $(n+1) (\delta \minn \rho(t,M)/2) >
  t$, then we may apply~\eqref{eq:5.15} and~\eqref{eq:5.16} to
  conclude that $\bb P^n[T_{n+1} \leq t] \leq \varepsilon$.  We have
  now shown that~\eqref{eq:5.14} holds.

  We are now essentially done.	It follows from~\eqref{eq:5.14} that
  there exists a unique probability measure $\bb P$ which agrees with
  $\bb P_n$ on $\scr C_{T_{n+1}}$ for all $n$
  (e.g.~\cite{stroock1979mdp} Theorem~1.3.5).  Moreover, we have shown
  that $X^{n+1}$ is a solution to the $(\ind{[0,T_{n+1}]} a(X^n),
  \ind{[0,T_{n+1}]} b(X^n))$-martingale problem starting at $(0,0)$
  under $\bb P$ for all $n$, and $T_n \rightarrow \infty$, $\bb
  P$-a.s., so we may conclude that $X$ is solution to the $(a(X),
  b(X))$-martingale problem starting at $(0,0)$ under $\bb P$.
  Finally, if $\Phat$ is another solution to the $(a(X),
  b(X))$-martingale problem starting at $(0,0)$ then it follows as in
  Lemma~6.6.4 in \cite{stroock1979mdp} that $\Phat$ must agree with
  $\bb P$ on $\scr C_{T_{n+1}}$ for each $n$.  As $\bb P$ is the
  unique measure with this property, we must have $\Phat = \bb P$.
\end{proof}

We now use this localization result to produce a global uniqueness result.

\begin{theorem} \label{thm:5.10} %
  Let $a: \bb R_+ \cross \bb R^d \rightarrow
  S^{d_0}_+$ and $b: \bb R_+ \cross \bb R^{d} \rightarrow \bb R^{d}$
  be measurable functions and suppose that
  \begin{alignat}{2}
    \label{eq:5.19}
    \inf_{s \in [0,T]} \inf_{\theta \in \bb R^d \setminus 0}
    \langle a(s,x) \theta, \theta \rangle / |\theta|^2 &> 0,
    &\qquad &\text{for all $T > 0$, $x\in \bb R^d$,}
    \\
    \label{eq:5.20}
    \lim_{y \rightarrow x} \sup_{s \in [0,T]}
    \| a(s,y) - a(s,x) \| &= 0,
    &\qquad &\text{for all $T > 0$, $x\in \bb R^d$,}
    \\
    \label{eq:5.21}
    b^{i}(s,x) &= 0, &\qquad &d_0 < i \leq d,\,
			(s,x) \in \bb R_+ \cross \bb R^d.
  \end{alignat}
  Further suppose that there exists a constant $N$ such that
  \begin{equation}
    \label{eq:5.22}
    \|a(s,x)\| + |b(s,x)|^2 \leq N (1 + |x|^2),
    \qquad \text{for all $(s,x) \in \bb R_+ \cross \bb R^d$.}
  \end{equation}
  Then the $(a(X), b(X) + BX)$-martingale problem is well-posed.
\end{theorem}
\begin{proof}
  Let $\varepsilon(\mu) = \varepsilon(B, \mu) > 0$ denote the constant
  obtained in Lemma~\ref{thm:5.8} as a function of $\mu$.  For each
  point $(s,x) \in \bb R_+ \cross \bb R^d$, we may choose $\mu(s,x) >
  0$ such that $a(t,x) \in S^{d_0}_{\mu(s,x)/2}$ when $|t-s| \leq 1$.
  This follows from~\eqref{eq:5.19} and~\eqref{eq:5.22}.  We may then
  use~\eqref{eq:5.20} to choose $r(s,x) \in (0,1)$ such that $a(t,y) \in
  S^{d_0}_{\mu(s,x)}$ and $\|a(t,y) - a(t,x)\| \leq
  \varepsilon(\mu(s,x))$ when $(t,y) \defeq (s-r(s,x),
  s+r(s,x)) \cross B^d_{r(s,x)}(x)$.

  We now let $\pi_{s,x}$ denote the unique projection onto the closed,
  convex set $[s-r(s,x)/2, s+r(s,x)/2] \cross \overline
  B^d_{r(s,x)/2}(x)$, and we define the functions $a_{s,x}(t,y) \defeq
  a\circ\pi_{s,x}(t,y)$, $c_{s,x}(t) \defeq a(\{s-r(s,x)/2\} \vee t
  \wedge \{s+r(s,x)/2\}, x)$, and $b_{s,x}(t,y) \defeq
  b(t,y)\circ\pi_{s,x}(t,y)$.  We then see that $a_{s,x}: \bb R_+
  \cross \bb R^d \rightarrow S^{d_0}_{\mu(s,x)}$, $c_{s,x}: \bb R_+
  \rightarrow S^{d_0}_{\mu(s,x)}$, and $b_{s,x}$ are bounded;
  $b^i_{s,x} = 0$ for $d_0 < i \leq d$; and $\|a_{s,x}(t,y) -
  c_{s,x}(t)\| \leq \varepsilon(\mu(s,x))$ for all $(t,y) \in \bb R_+
  \cross \bb R^d$, so we may invoke Theorem~\ref{thm:5.8} to conclude
  that the $(a_{s,x}(X), b_{s,x}(X) + BX)$-martingale problem is
  well-posed.  We then define the open cover $\{G_{s,x}\}_{(s,x) \in
    \bb R_+\cross\bb R^d}$ where $G_{s,x} \defeq (s-r(s,x)/2,
  s+r(s,x)/2) \cross B^d_{r(s,x)/2}(x)$, and apply the
  previous localization result to conclude that the $(a,b)$-martingale
  problem is well-posed.
\end{proof}

The results provided up to this point provide very little flexibility
regarding the drift of the last $d-d_0$ components of the process: The
drift must take the form $b''(t,x) = B'' x$ where $B''$ denotes the
last $d-d_0$ rows of a matrix $B$ which satisfies the structure
conditions given in Section~\ref{sec:2}.  As a result, we cannot apply
the previous theorem to an SDE of the form:
\begin{equation} \label{eq:5.23}
  \left\{
  \begin{aligned}
    \diff X_t &= \sigma(t,X_t,Y_t) \, \diff W_t,
    \\
    \diff Y_t &= (X_t + Y_t) \, \diff t.
  \end{aligned}
  \right.
\end{equation}
However, if we define the process $Z \defeq X + Y$, then we see that
the process $(Z,Y)$ solves the SDE:
\begin{equation} \label{eq:5.24}
  \left\{
  \begin{aligned}
    \diff Z_t &= Z_t \, \diff t + \sigmahat(t,Z_t,Y_t) \, \diff W_t,
    \\
    \diff Y_t &= Z_t \, \diff t,
  \end{aligned}
  \right.
\end{equation}
where $\sigmahat(t,z,y) = \sigma(t,z-y,y)$.  Written in this form, the
drift now satisfies the required structure condition with
$B=\begin{bmatrix} 0 & 0 \\ 1 & 0 \end{bmatrix}$.  Moreover, the set
of weak solutions to the SDE~\eqref{eq:5.23} are in one-to-one
correspondence with the set of weak solutions to the
SDE~\eqref{eq:5.24}.  In particular, if $\sigmahat$ satisfies
conditions which ensure existence and uniqueness in law for the
SDE~\eqref{eq:5.24}, then existence and uniqueness in law also hold
for the SDE~\eqref{eq:5.23}.  We now generalize this observation
slightly to obtain a second local uniqueness result.  Recall that $D
b''$ denotes the $d_1 \cross (1+d)$ Jacobian matrix of the function
$b''$ when $b''$ takes values in $\bb R^{d_1}$.

\begin{lemma} \label{thm:5.11} %
  Suppose that $d = d_0 + d_1$ with $d_1 \leq d_0$.  Let $a : \bb R_+
  \cross \bb R^{d} \rightarrow \Smu$, $b': \bb R_+ \cross \bb R^{d}
  \rightarrow \bb R^{d_0}$, and $c : \bb R_+ \rightarrow \Smu$ be
  bounded, measurable functions, let $b'' \in C^{1,2,1}(\bb R_+ \cross
  \bb R^{d_0}\cross \bb R^{d_1}; \bb R^{d_1})$ denote a function with
  bounded derivatives, and let $b: \bb R_+ \cross \bb R^{d}
  \rightarrow \bb R^{d}$ denote the function $b(s,x) = (b'(s,x),
  b''(s,x))$.  Let $A_0 \in \bb M^{d_1\cross 1}$, $A_1 \in \bb
  M^{d_1\cross d_0}$ and $A_2 \in \bb M^{d_1\cross d_1}$ be matrices,
  and suppose that $A_1$ is of rank $d_1$.  Finally, let $A = [A_0 \;
  A_1 \; A_2] \in \bb M^{d_1 \cross (1+d)}$ denote the matrix obtained
  by appending the columns of $A_1$ and $A_2$ to $A_0$.  Then there
  exists a constant $\varepsilon = \varepsilon(A, \mu) > 0$ such that
  there is at most one solution to the $(a(X), b(X))$-martingale
  problem starting at $(s,x)$ when
  \begin{equation}
    \label{eq:5.25}
    \|a(t,y) - c(t)\| + \|D b''(t,y) - A \| \leq \varepsilon,
    \qquad \text{for all $(t,y) \in \bb R_+ \cross \bb R^d$.}
  \end{equation}
\end{lemma}
\begin{proof}
  As in the proof of Lemma~\ref{thm:5.8}, we may assume that $b' = 0$.
  We first show that the result holds when $d_1 = d_0$.  Until further
  notice, we will write a generic point $x \in \bb R^d$ in the form $x
  = (x',x'') \in \bb R^{d_0} \cross \bb R^{d_1}$, and we will let
  $\pi''$ denote the projection $\pi''(x',x'') = x''$.  Define
  \begin{align*}
    \Ahat &\defeq \begin{bmatrix} 1 & 0 & 0 \\
				  A_0 & A_1 & A_2 \\
				  0 & 0 & I^{d_0}
		 \end{bmatrix} \in M^{(1+d)},
    \\
    \Bhat &\defeq \begin{bmatrix}
		  0 & 0 \\
		  I^{d_0} &0
		\end{bmatrix} \in \bb M^{d},
  \end{align*}
  where $0$ denotes a matrix of zeros whose dimensions vary at each
  occurrence, and set $\muhat = 4 (\|A_1\|^2 + \|\Ahat\inv\|^{2})
  \mu$.  The assumption that $A_1$ is of full rank ensures that
  $\Ahat$ is invertible and the matrix $\Bhat$ clearly satisfies the
  structure conditions given in Section~\ref{sec:2}.  We then let
  $\varepsilon_1 = \varepsilon_1(\Bhat, \muhat)$ denote the constant
  obtained in Lemma~\ref{thm:5.8} and set
  \begin{equation*}
    \varepsilon = \min\biggl\{ \|A_1\|, \;
			       \frac{1}{2 \|\Ahat\inv\|},\;
			       \frac{\varepsilon_1}
				    {6 (1+\mu)\|A_1\|(1+\|A_1\|)}
			       \biggr\}.
  \end{equation*}
  Let $\ahat : \bb R_+ \cross \bb R^d \rightarrow S^{d_0}_+$ and
  $\bhat : \bb R_+ \cross \bb R^d \rightarrow \bb R^{d}$ denote the
  functions
  \begin{align*}
    \ahat(s,x)
    &= D_{x'}b^{\prime\prime}(s,x)\; a(s,x)\; D^T_{x'}b^{\prime\prime}(s,x),
    \\
    \bhat^i &= \indd{1 \leq i \leq d_0} \bigl\{
      \partial_s b^{\prime\prime,i}
	+ \half \tr(D^2_{x'} b^{\prime\prime,i} a)
	+ \langle D_{x''} b^{\prime\prime,i},\, b^{\prime\prime}
	\rangle\bigr\},
    \qquad 1 \leq i \leq d,
  \end{align*}
  where the arguments of $\bhat$ have been suppressed.	The
  derivatives of $b''$ are bounded, so $\bhat$ satisfies a
  linear growth condition.  Finally, let $\chat : \bb R_+ \rightarrow
  S^{d_0}_+$ denote the function $\chat(t) = A_1 \, c(t) \, A^T_1$.

  We now assume that~\eqref{eq:5.25} holds and show that $\ahat$ and
  $\chat$ take values in $S^{d_0}_{\muhat}$.  If $z \in \bb R^{d_0}$
  and $(s,x) \in \bb R_+ \cross \bb R^d$, then
  \begin{equation*}
    \frac{|z|}{\|A\inv_1\|} \leq |A^T_1 z|
    \leq |D^T_{x'}b^{\prime\prime}(s,x) \, z|
	 + \| D^T_{x'}b^{\prime\prime}(s,x) - A^T_1 \| \, |z|.
  \end{equation*}
  It then follows from our choice of $\varepsilon$ that
  \begin{equation*}
    |z| / (2 \|A\inv_1\|)
    \leq |D^T_{x'}b^{\prime\prime}(t,x) \, z|
    \leq 2 \|A_1\| |z|.
  \end{equation*}
  Moreover, as $\langle a(s,x) z, z\rangle \in [\mu\inv|z|^2, \mu
  |z|^2]$ and $\langle c(s) z, z\rangle \in [\mu\inv |z|^2, \mu
  |z|^2]$ by assumption, we see that $\ahat$ and $\chat$ take values
  in $S^{d_0}_{\muhat}$.  We also have
  \begin{align*}
    \|\ahat(s,x) - \chat(s)\|
    &\leq \|D_{x'}b^{\prime\prime}(s,x)\| \|a(s,x)\|
	 \|D_{x'}b^{\prime\prime}(s,x) - A_1\|
    \\	  &\qquad
	 + \|D_{x'}b^{\prime\prime}(s,x)\| \|a(s,x) - c(s)\| \|A_1\|
    \\	  &\qquad
	 + \|D_{x'}b^{\prime\prime}(s,x) - A_1\| \|c(s)\| \|A_1\|
    \\
    &\leq 6 (1+\mu)\|A_1\|(1+\|A_1\|) \varepsilon,
  \end{align*}
  so $\|\ahat(s,x) - \chat(s)\| \leq \varepsilon_1$ for all $(s,x) \in
  \bb R_+ \cross \bb R^d$.

  Now let $\bb P_1$ and $\bb P_2$ denote two solutions to the $(a(X),
  b(X))$-martingale problem starting at $(s,x) = (s,x',x'')$,
  let $f : \bb R_+ \cross \bb R^d \rightarrow \bb R_+ \cross \bb
  R^d$ denote the function $f(s,x',x'') = (s, b''(s,x',x''), x'')$, and
  let $Y : \bb R_+ \cross C(\bb R_+; \bb R^d)
  \rightarrow \bb R^d$ denote the process
  \begin{align*}
    Y_t = \indd{t < s} \bigl(b''(s,x), x''\bigr)
	  + \indd{t \geq s} \bigl(b''(t,X_t), \pi''(X_t)\bigr).
  \end{align*}

  We now show that $f$ maps $\bb R_+ \cross \bb R^d$ onto itself and
  admits a continuous inverse. To see this, fix any $(t_0,y_0) \in \bb
  R_+\cross \bb R^d$ and define the function $\phi_{t_0,y_0}(s,x) =
  \Ahat\inv (t_0,y_0) - \Ahat\inv\{f(s,x) - \Ahat (s,x)\}$.  Then
  $\phi_{t_0,y_0}$ maps $\bb R_+\cross\bb R^d$ into $\{t_0\}\cross \bb
  R^d$.  It follows directly from~\eqref{eq:5.25} that $\|D f(s,x) -
  \Ahat\| \leq \varepsilon \leq 1 / (2 \|\Ahat\inv\|)$ for all $(s,x)
  \in \bb R_+\cross \bb R^d$, so $\phi_{t_0,y_0}$ is a strict
  contraction, and $(s,x)$ is a fixed point of $\phi_{t_0,y_0}$ if and
  only if $f(s,x) = (t_0, y_0)$.  By varying $(t_0,y_0)$ and arguing
  as in the Inverse Function Theorem, we may conclude that $f$ admits
  a continuous inverse $f\inv : \bb R_+ \cross \bb R^d \rightarrow \bb
  R_+ \cross \bb R^d$.

  We now observe that $(t,Y_t) = f(t, X_t)$, so it follows from Ito's
  Lemma that $Y$ solves the $(\ahat \circ f\inv(Y),\, \bhat
  \circ f\inv(Y) + \Bhat Y)$-martingale problem starting at $f(s,x)$
  under both $\bb P_1$ and $\bb P_2$.  We may then apply
  Lemma~\ref{thm:5.8} to conclude that $Y$ has the same law under $\bb
  P_1$ and $\bb P_2$.  As $(t, X_t) = f\inv(t, Y_t)$, we see that $X$
  also has the same law under both $\bb P_1$ and $\bb P_2$, but then
  we must have $\bb P_1 = \bb P_2$.

  We now show that the result holds when $d_1 < d_0$.  We will do this
  by appending additional components to the process and then ignoring
  them.  Let $\Atilde_0 \in \bb M^{d_0\cross 1}$ denote the matrix
  obtained from $A_0$ be adding $d_0 - d_1$ zeros, let $\Atilde_1 \in
  \bb M^{d_0}$ denote a matrix obtained by appending $d_0 - d_1$
  linearly independent rows to $A_1$, let $\Atilde_2 \in \bb M^{d_0}$
  denote the matrix such that $\Atilde_2^{ij} = A_2^{ij}$ if $1 \leq
  i,j \leq d_1$ and $\Atilde_2^{ij} = 0$ otherwise, and set $\Atilde =
  [\Atilde_0\;\Atilde_1\;\Atilde_2]$.  Finally, let $\varepsilon =
  \varepsilon(\Atilde, \mu)$ denote the constant that was obtained in
  the previous case.

  For the remainder of the lemma, we will write a generic point in $x
  \in \bb R^{2d_0}$ in the form $x = (x', x'') \in \bb R^{d} \cross
  \bb R^{d_0 - d_1}$, so $x''$ denotes the extra coordinates that we
  are adding to place ourselves in the previous case.  We will let
  $\pi$ denote the projection $\pi(s,x',x'') = (s,x')$ that removes
  these extra coordinates.  Let $\bb P_1$ and $\bb P_2$ denote
  solutions to the $(a(X), b(X))$-martingale problem starting at
  $(s,x) \in \bb R_+\cross \bb R^{d}$ and assume that~\eqref{eq:5.25}
  holds.  Set $\atilde = a \circ \pi$, set $\btilde' = b' \circ \pi$,
  let $\btilde'' : \bb R_+ \cross \bb R^{2d_0} \rightarrow \bb
  R^{d_0}$ denote the function
  \begin{equation*}
    \btilde^{\prime\prime,i}(s,x)
    = \indd{1 \leq i \leq d_1} b^{\prime\prime,i} \circ \pi(s,x)
      + \indd{d_1 < i \leq d_0}
	\sum_{j=1}^{d_0} \Atilde^{ij}_1 x^{j},
    \qquad 1 \leq i \leq d_0,
  \end{equation*}
  and let $\btilde : \bb R_+ \cross \bb R^{2d_0} \rightarrow \bb
  R^{2d_0}$ denote the function $\btilde(s,x) = (\btilde'(s,x),
  \btilde''(s,x))$.  Now let $Z : \bb R_+ \cross C(\bb R_+; \bb R^d)
  \rightarrow \bb R^{2d_0}$ denote the process
  \begin{equation*}
    Z^i_t = \indd{1 \leq i \leq d} X^i_t
	     + \indd{d < i \leq 2d_0}
	       \sum_{j=1}^{d_0} \int_s^{s \vee t}
		 \Atilde^{(i-d_0),j} X^{j}_u \, \diff u,
    \qquad 1 \leq i \leq 2d_0.
  \end{equation*}
  Its not hard to check that $Z$ is a solution to $(\atilde(Z),
  \btilde(Z))$-martingale problem starting from $(s, x, 0) \in \bb
  R_+\cross \bb R^{d}\cross \bb R^{d_0-d_1}$ under both $\bb P_1$ and
  $\bb P_2$.  It is also easy to see that~\eqref{eq:5.25} implies
  \begin{equation*}
    \|\atilde(t,x) - \ctilde(t)\|
    + \|D \btilde''(t,x) - \Atilde \| \leq \varepsilon,
    \qquad \text{for all $(t,x) \in \bb R_+ \cross \bb R^{2d_0}$.}
  \end{equation*}
  As a result, we may apply the previous case to conclude that $Z$ has
  the same law under $\bb P_1$ and $\bb P_2$, but then $\bb P_1 = \bb
  P_2$.
\end{proof}

We now give a simple approximation lemma that we will need to
extend the previous result.

\begin{lemma}
  Let $f \in C^1(\bb R^d)$ and fix any $x \in \bb R^d$ and
  $\varepsilon > 0$.  Then there exists a constant $r > 0$ and a
  function $g \in C^1(\bb R^d)$ such that $g = f$ on $B_r(x)$ and $\|Dg(y)
  - Df(x)\| \leq \varepsilon$ for all $y \in \bb R^d$.
\end{lemma}
\begin{proof}
  Let $\eta \in C^\infty_K(\bb R; [0,1])$ denote a function with
  $\eta(z) = 1$ when $|z| \leq 1$, $\eta(z) = 0$ when $|z| \geq 3$,
  and $|\eta'(x)| \leq 1$ everywhere.  By translation and the addition
  of an affine function, we may assume that $x = Df(x) = 0$ and $f(x)
  = 0$.	 We now choose $r$ such that $|f(x)| / |x| \leq \varepsilon /
  64 \sqrt{d}$ and $|Df(x)| \leq \varepsilon / 2\sqrt{d}$ when $|x| \leq
  4r$, and set $g(x) = \eta(|x|^2/r^2) f(x)$.  It is clear that $g(x)
  = f(x)$ on $B_r(x)$ and $Dg(x) = 0$ when $|x| \geq 4r$.  But when
  $|x| \leq 4 r$, we have
  \begin{equation*}
    |\partial_i g(x)| \leq 2 |x| |f(x)| / r^2 + |\partial_i f(x)|
    < \varepsilon / \sqrt{d},
  \end{equation*}
  so $g$ possesses the desired properties.
\end{proof}

\begin{corollary} \label{thm:5.13} %
  Let $f \in C^{1,2,1}(\bb R_+\cross\bb R^{d_0}\cross\bb R^{d_1})$,
  let $(s,x) \in \bb R_+\cross\bb R^{d_0 + d_1}$, and let $\varepsilon
  > 0$.	 Then there exists a function $g \in C^{1,2,1}(\bb
  R_+\cross\bb R^{d_0}\cross\bb R^{d_1})$ with bounded derivatives and
  a constant $r > 0$ such that
  \begin{alignat*}{2}
    g(t,y) = f(t,y),&
    &\qquad &\text{for all $(t,y) \in (s-r,s+r) \cross B^{d_0+d_1}_r(x)$},
    \\
    |Dg(t,y) - Df(s,x)| \leq \varepsilon,&
    &\qquad &\text{for all $(t,y) \in \bb R_+ \cross \bb R^{d_0 + d_1}$.}
  \end{alignat*}
\end{corollary}
\begin{proof}
  We may construct a function $\fhat \in C^{1,2,1}(\bb R\cross\bb
  R^{d_0}\cross\bb R^{d_1})$ with $\fhat(t,y) = f(t,y)$ when $t
  \geq 0$ (e.g. \cite{gilbarg2001epde} 6.37), so the corollary
  follows immediately from proof of the previous lemma.
\end{proof}

We now have all the tools that we need for the final result of
the paper.

\begin{theorem} \label{thm:5.14} %
  Let $d = d_0 + d_1$ with $d_1 \leq d_0$ and let $x = (x',x'') \in
  \bb R^{d_0}\cross\bb R^{d_1}$ denote the generic point in $\bb R^d$.
  Let $a: \bb R_+ \cross \bb R^d \rightarrow S^{d_0}_+$ and $b': \bb
  R_+ \cross \bb R^{d} \rightarrow \bb R^{d_0}$ be measurable
  functions, let $b'' \in C^{1,2,1}(\bb R_+\cross \bb R^{d_0}\cross\bb
  R^{d_1}; \bb R^{d_1})$, and let $b : \bb R_+\cross\bb R^d
  \rightarrow \bb R^d$ denote the function $b(s,x) = (b'(s,x),
  b''(s,x))$.
  Suppose that
    \begin{alignat}{2}
    \label{eq:5.26}
    \inf_{s \in [0,T]} \inf_{\theta \in \bb R^d \setminus 0}
    \langle a(s,x) \theta, \theta \rangle / |\theta|^2 &> 0,
    &\qquad &\text{for all $T > 0$, $x\in \bb R^d$,}
    \\
    \label{eq:5.27}
    \lim_{y \rightarrow x} \sup_{s \in [0,T]}
    \| a(s,y) - a(s,x) \| &= 0,
    &\qquad &\text{for all $T > 0$, $x\in \bb R^d$.}
  \end{alignat}
  Further suppose that there exists a constant $N$ such that
  \begin{equation} \label{eq:5.28} %
    \|a(s,x)\| + |b(s,x)|^2 \leq N (1 + |x|^2),
    \quad \text{for all $(s,x) \in \bb R_+ \cross \bb R^d$,}
  \end{equation}
  and that
  \begin{equation} \label{eq:5.29} %
    \text{ $D_{x'} b''(s,x)$ is of rank $d_1$  for all $(s,x) \in \bb R_+
	   \cross \bb R^d$.}
  \end{equation}
  Then the $(a(X), b(X))$-martingale problem is well-posed.
\end{theorem}

\begin{proof}
  For each point $(s,x) \in \bb R_+ \cross \bb R^d$, we may again
  choose $\mu(s,x) > 0$ such that $a(t,x) \in S^{d_0}_{2\mu(s,x)}$
  when $|t-s| \leq 1$.  Now let $\varepsilon(A, \mu) > 0$ denote the
  constant obtained in Lemma~\ref{thm:5.11} and set $\delta(s,x)
  \defeq \varepsilon(Db''(s,x), \mu(s,x)) / 2$.  It follows from
  Corollary~\ref{thm:5.13} and condition~\eqref{eq:5.27} that we may
  find a function $b''_{s,x} \in C^{1,2,1}(\bb R_+\cross \bb
  R^{d_0}\cross\bb R^{d_1}; \bb R^{d_1})$ with bounded derivatives and
  a radius $r(s,x) > 0$ such that:
  \begin{itemize}[topsep=3pt]
  \item $a(t,y) \in S^{d_0}_{\mu(s,x)}$, $\|a(t,y) - a(t,x)\| \leq
    \delta(s,x)$, and $b''_{s,x}(t,y) = b''(t,y)$ when $(t,y)
    \in (s-r,s+r)\cross B^d_r(x)$, and
  \item $\|Db''_{s,x}(t,y) - Db''(s,x)\| \leq \delta(s,x)$ for
    all $(s,x) \in \bb R_+ \cross \bb R^d$.
  \end{itemize}

  We now set $G_{s,x} \defeq (s-r(s,x)/2, s+r(s,x)/2)\cross
  B^d_{r(s,x)/2}(x)$, let $\pi_{s,x}$ denote the unique projection
  onto the closure of $G_{s,x}$, and define the functions
  \begin{align*} 
    a_{s,x}(t,y) &\defeq a\circ\pi_{s,x}(t,y),
    \\
    c_{s,x}(t) &\defeq a(\{s-r(s,x)/2\} \vee t \wedge \{s+r(s,x)/2\}, x),
    \\
    b'_{s,x}(t,y) &\defeq b'(t,y)\circ\pi_{s,x}(t,y),
    \\
    b_{s,x}(t,y) &\defeq (b'_{s,x}(t,y), b''_{s,x}(t,y)).
  \end{align*}
  The functions $a_{s,x}$ and $b''_{s,x}$ are continuous in the
  spacial variable and satisfy a linear growth condition, the function
  $a_{s,x}$ is uniformly positive-definite, and the function
  $b'_{s,x}$ is bounded.  As a result, the existence of a solution to
  the $(a_{s,x}, b_{s,x})$-martingale problem for each initial
  condition follows from Lemma~\ref{thm:5.6} and Girsanov's Theorem.

  The functions $a_{s,x}$ and $c_{s,x}$ take values in
  $S^d_{\mu(s,x)}$ and
  \begin{align*}
    \|a_{s,x}(t,y) - c_{s,x}(t)\| + \|Db_{s,x}''(t,y) - Db''(s,x)\|
    \leq \varepsilon(Db''(s,x), \mu(s,x)),
  \end{align*}
  for all $(t,y) \in \bb R_+\cross\bb R^d$.
  It then follows from Lemma~\ref{thm:5.11} that the $(a_{s,x},
  b_{s,x})$-martingale problem is well-posed.  Finally, we have
  $a_{s,x} = a$ and $b_{s,x} = b$ on $G_{s,x}$ and
  $\{G_{s,x}\}_{(s,x) \in \bb R_+\cross\bb R^d}$ is an open cover of
  $\bb R_+\cross\bb R^d$, so we may invoke Theorem~\ref{thm:5.9} to
  conclude that that the $(a,b)$-martingale problem is well-posed.
\end{proof}


\def\cprime{$'$} \def\polhk#1{\setbox0=\hbox{#1}{\ooalign{\hidewidth
  \lower1.5ex\hbox{`}\hidewidth\crcr\unhbox0}}} \def\cprime{$'$}

\end{document}